\documentclass{siamltex704} 
 
\usepackage{amsmath,amssymb}
\usepackage{amsfonts}
\usepackage{exscale} 
\usepackage{graphicx} 
\usepackage{url}  
\usepackage[pdftex]{thumbpdf}      
\usepackage[pdftex,                
pdfstartview=FitH,%
bookmarks=true,
bookmarksnumbered=true,
bookmarksopen=true,%
hypertexnames=false,
breaklinks=true,
  colorlinks=true,%
  linkcolor=blue,anchorcolor=blue,%
  citecolor=blue,filecolor=blue,%
  menucolor=blue]%
{hyperref} 
 
\usepackage{etoolbox}  
\newtoggle{LONG} 
\toggletrue{LONG} 
 
\makeatletter 
\def\@cite#1#2{{\rm [}{{\rm#1}\if@tempswa , #2\fi}{\rm ]}} 
\makeatother

\graphicspath{{./FIGS/}}

\newtheorem{algorithm}{Algorithm}[section]

\newcommand{\eq}[1]{\begin{equation}\label{#1}} 
\newcommand{\en}{\end{equation}}

\def\IR{\mathbb{R}} 
 
\def\Span{\mbox{\texttt{span}}} 
\def\Range{\mbox{\texttt{range}}} 
\def\calS{\mathcal{S}}  
  
\def\calU{\mathcal{U}}  
\def\calV{\mathcal{V}}  
\def\calZ{\mathcal{Z}}  
\def\NNZ{\mbox{\texttt{nnz}}}

\title{ 
Fast updating algorithms for latent semantic indexing\thanks{This work was supported by 
NSF, under grant CCF-1318597, and by the Minnesota Supercomputing Institute.}}

\author{ 
Eugene Vecharynski\thanks{Computational Research Division,  
Lawrence Berkeley National Laboratory, Berkeley, CA 94720 
({\tt eugene.vecharynski@gmail.com).}} 
\and  Yousef Saad\thanks{ 
Department of Computer Science and 
Engineering, University of Minnesota, 200 Union Street S.E., 
Minneapolis, MN 55455, USA ({\tt saad@cs.umn.edu}).
} 
}

\begin{document} 
\maketitle 
 
\setcounter{page}{1}

\begin{abstract} 
This paper discusses a few algorithms for updating the approximate 
Singular Value Decomposition (SVD) in the context of information retrieval 
by Latent Semantic Indexing (LSI) methods. A unifying framework is  
considered which is based on Rayleigh-Ritz projection methods. First, 
a  Rayleigh-Ritz approach for the SVD is discussed and it is then  
used to interpret the Zha--Simon algorithms [SIAM J. Scient. Comput. 
vol. 21 (1999), pp. 782-791]. This viewpoint leads to a few alternatives  
whose  goal is to reduce computational cost and storage requirement by projection 
techniques that utilize subspaces of much smaller dimension. 
Numerical experiments 
show that the proposed algorithms yield accuracies comparable 
to
those obtained from standard ones at a much lower  computational cost. 
\end{abstract}

 

%
 
\pagestyle{myheadings} 
\thispagestyle{plain} 
\markboth{EUGENE VECHARYNSKI AND YOUSEF SAAD}{FAST UPDATING ALGORITHMS FOR LSI} 

\section{Introduction}\label{sec:intro} 
Latent Semantic Indexing (LSI), introduced in~\cite{LSI:90}, is a well-established  
text mining technique that aims at finding documents in a   given collection that are relevant to a user's 
query. The method is a variation 
 of the Principal Component Analysis (PCA)~\cite{Bishop-book}, where  
the multidimensional text dataset is  projected to a  
low-dimensional subspace. 
When properly defined, this subspace captures  the essence of 
 the original data.  
In the projected space, \textit{semantically} similar documents tend 
to be close to each other in a certain measure, 
which allows to compare  them according to their \emph{latent semantics} 
rather  than a straightforward word matching.  
 
 
LSI can be  viewed as an extension of  the \textit{vector space} model 
for  Information  Retrieval  (IR)~\cite{Salton.McGill:83}. As such, it begins 
with  a  preprocessing   phase   (see,  e.g,~\cite{Baeza.Ribeiro:99, 
 TMG-tool})  to    summarize the  whole  text 
collection in the  $m$-by-$n$ \textit{term-document} matrix $A$, where 
$m$  and $n$  are the  total  numbers of  terms and  documents in  the 
collection, respectively. 
Thus, each column of $A$ represents a separate document, where nonzero entries are the weights, or essentially the frequencies,  
of the terms occurring in this document.  
For the discussion of the available term weighting schemes we refer the reader to~\cite{Kolda.OLeary:98}.

We consider a  widely  used and standard implementation of LSI that is  
based on the partial Singular Value Decomposition   
(SVD)~\cite{GVL-book} of the term-document matrix.  
In this case, LSI resorts  
to calculating the singular triplets $(\sigma_j, u_j, v_j)$ 
associated with the $k$ largest singular values $\sigma_1 \geq \sigma_2 \geq \ldots \geq \sigma_k \geq 0$ of $A$. 
Throughout, we call these $k$ triplets the \textit{dominant} singular triplets. 
The \textit{left} singular vectors $u_j$ are then used to construct the low-dimensional subspace for data 
projection and, along with $\sigma_j$ and the \textit{right} singular vectors $v_j$, to evaluate the relevance scores.   
We note that a number of ``SVD avoiding'' LSI algorithms have been proposed in recent years,  
e.g.,~\cite{Blom.Ruhe:05, ChenSaad:FiltLan, Dhillon.Modha:01, KokSaad-IR, Kolda.OLeary:98}, for example, by replacing the SVD by the Lanczos decomposition. 
We will briefly discuss one such alternative based on using Lanczos vectors. 

Given a user's query $q$, formalized by a vector of size $m$, i.e., regarded  
as a document, the associated vector $r$ of $n$ relevance scores is evaluated by 
\begin{equation}\label{eqn:scores} 
r = \mbox{diag}\left( \gamma_1, \ldots, \gamma_n \right) \left(V_k \Sigma_k^{1-\alpha} \right) \Sigma_k^{\alpha} U_k^T  q.  
\end{equation} 
Here, $\Sigma_k = \mbox{diag}\left\{ \sigma_1, \sigma_2, \ldots, \sigma_k \right\} \in \mathbb{R}^{k \times k}$; 
$U_k = [u_1, \ldots, u_k] \in \mathbb{R}^{m \times k}$ and $V_k = [v_1, \ldots, v_k] \in \mathbb{R}^{n \times k}$ 
are the matrices of the orthonormal left and right singular vectors, respectively. The diagonal elements $\gamma_j$ 
are chosen to normalize the rows of $V_k \Sigma_k^{1-\alpha}$ so that each row has a unit norm.  
The scalar $\alpha$ is a splitting parameter and has no affect on ranking if the  
normalization is disabled ($\gamma_j = 1$); see, e.g.,~\cite{Kolda.OLeary:98} for a detailed discussion. 
  
The $j$-th entry of $r$, denoted by $r(j)$, quantifies the relevance between the $j$-th document 
and the query.  
The documents with the highest relevance scores are returned to the user in response to the query $q$.  
Note that, for example, in the case where $\alpha = 0$, $r(j)$ is the (scaled) cosine of the angle  
between the $j$-th projected document $U_k^T a_j$ and the projected query $U_k^T q$.  
 
In practical applications, where the amount of data tends to be extremely large, 
the implementation of LSI faces two major difficulties. 
The first difficulty is  
the requirement to compute the dominant singular triplets of a very large 
matrix, a problem that has been relatively well investigated.  
Possible solutions include 
invoking iterative singular value solvers,  
e.g.,~\cite{SVDPACK,Hernandez.Roman.Tomas:08},  
that take advantage of  sparsity and fast matrix-vector products.  
Other solutions leverage specific spectral properties  
of the term-document matrices, and rely on incremental or  
divide-and-conquer techniques;  
e.g.,~\cite{Zha.Simon:99, ChenSaad:Mlev-08}. 
 
The second computational difficulty of LSI is related to the fact that 
document collections are dynamic, i.e., the term-document matrices are subject to repeated updates. 
Such  updates result from  adding, e.g.,  new 
documents or terms  to the collection.  In the  language of the vector 
space model, this  translates into adding new columns  or rows to $A$. 
Another type of  update is when the term  weights are corrected, which 
corresponds to  modifying entries of the  term-document matrix.  Thus, 
in order to  maintain the quality of the  subsequent query processing, 
the available  singular triplets should be  accordingly modified after 
each  update.   A  straightforward  solution  to this  problem  is  to 
recompute  the partial SVD  of the  updated term-document  matrix from 
scratch.   However, even  with the  most sophisticated  singular value 
solvers, this  naive approach  is not practical  as it  is exceedingly 
costly  for  realistic  large-scale  text collections.   Therefore,  a 
critical question is how  to \textit{update} the available $\Sigma_k$, 
$U_k$, and  $V_k$ without fully recomputing the  high-cost partial SVD 
of the modified matrix, so that the retrieval quality is not affected. 
 
This  paper  addresses  this  specific question.   It starts by  
revisiting   the   well-known    updating   algorithms   of   Zha   and 
Simon~\cite{Zha.Simon:99},  currently the 
state-of-the-art approach for the LSI updating problem. 
Specifically, the paper interprets these schemes as Rayleigh-Ritz
projection procedures. A by-product of this viewpoint is a
min-max type characterization of  the Ritz singular values obtained in the
process. On the practical side, this projection viewpoint 
unravels  a certain redundancy in the computations, showing
that it is possible to further improve the efficiency of the techniques  
without sacrificing retrieval quality.  
Based on these findings, we propose a family of new updating algorithms which can be substantially faster 
and less storage-demanding than the methods in~\cite{Zha.Simon:99}.  

The motivation  for the present  work comes from the  observation that 
the  methods  in~\cite{Zha.Simon:99}   (reviewed  in  more  detail  in 
Section~\ref{sec:lsi_proj})  rely on the  SVD of  a $(k+p)$-by-$(k+p)$ 
dense matrix  and orthogonalization  of $p$ ($m$-  or $n$-dimensional) 
vectors, where $p$ denotes the size of the update, i.e., 
the number of added columns, rows, or corrected terms.  
In particular, this suggests that the computational complexity of the overall updating  
procedure scales cubically with respect to $p$. 
 
While the effect of the cubic scaling is marginal for smaller document collections,  
where the update sizes are typically given by only a few terms or documents, 
the situation becomes different for large-scale datasets.             
In this  case, even if $p$ is  a tiny fraction of  terms or documents, 
its  (cubed)  value may  be  large  enough  to noticeably  affect  the 
efficiency  of  the  updating   methods.   In  other  words,  for  $p$ 
sufficiently large, the computational costs associated with the SVD of 
a $(k+p)$-by-$(k+p)$  matrix and orthogonalization  (QR decomposition) 
of  $p$ vectors  become non-negligible  and may  dominate  the overall 
updating procedure. 
 
Another context in which  larger updates are to be processed 
can be found in the recent  
works~\cite{Tougas.Spiteri:08, Mason.Spiteri:07}, 
where the authors suggest to postpone  invoking 
the updating schemes from~\cite{Zha.Simon:99}  
until the update size becomes sufficiently large. In between the updates,  
a fast \textit{folding-in} procedure~\cite{BeDuBri:95, Berry.Drmac.Jessup:99} is performed,
which can be viewed as a form of PCA out-of-sample embedding~\cite{Maaten.Postma.Herik:09} applied 
in the context of LSI.
Such a combination of folding-in and updates,  
called \textit{folding-up}, has been shown to yield 
a substantial reduction in the time spent for 
updating without a significant loss in the retrieval accuracy.

To adapt their algorithms to the cases of large $p$, the authors  
of~\cite{Zha.Simon:99} suggest splitting the current update into 
a series of smaller sequential updates, and performing 
the whole updating procedure in an incremental fashion.  
While this approach indeed leads to memory savings, 
it requires more time to complete the \textit{overall} updating task  
than to perform the whole update at once.  
This can be seen, e.g., from Table~$4.2$  
in the original paper~\cite{Zha.Simon:99}, after multiplying the reported average  
CPU times per update by the number of updates.  
A similar observation has been made in~\cite[Table 1]{Tougas.Spiteri:08}.  
Additionally, as has also been observed in~\cite{Tougas.Spiteri:08},  
breaking a given update into a sequence of smaller updates   
can potentially lead to a faster deterioration of the retrieval accuracy.

The updating algorithms proposed in this paper require orthonormalizing sets of significantly fewer vectors  
than those in~\cite{Zha.Simon:99} and rely on the SVD of much smaller matrices. 
As a result, the new schemes are less sensitive to the increase in the update sizes.  
As shown in our experiments, the presented algorithms significantly reduce the runtime, whereas  
the retrieval accuracy is not affected. 
%
%
 
Finally, let  us recall  that the methods  in~\cite{Zha.Simon:99} were 
introduced as a solution to the problem of the deteriorating retrieval 
accuracy   exhibited by   existing   methods~\cite{BeDuBri:95, OBrien:94}  
in the mid-1990s. 
This solution  essentially traded the  
SVD of a $k$-by-$k$ matrix in~\cite{BeDuBri:95, OBrien:94} for  
the above mentioned orthonormalization of a set of 
$p$ extra vectors plus a $(k+p)$-by-$(k+p)$ SVD. 
The updating schemes introduced in this work can be viewed as a compromise between~\cite{BeDuBri:95, OBrien:94}  
and \cite{Zha.Simon:99}, where the runtime resembles that of the former  
while the retrieval accuracy 
is comparable to the latter.

The rest of the  paper is organized as follows. Projection methods for the SVD  
are reviewed in Section~\ref{sec:proj} followed by  a discussion of  
their applications to LSI in Section~\ref{sec:lsi_proj}. In  
Section~\ref{sec:new_updates} a number of alternative algorithms are 
presented with a goal of reducing cost.  
Section~\ref{sec:num} presents numerical experiments to test the various 
methods introduced, and  
Section~\ref{sec:concl} concludes the paper with a few remarks. 
 
\section{Projection methods for singular value problems}\label{sec:proj} 
It will be useful to explore projection methods for the singular value problem 
in order to understand the mechanisms at play when updating the SVD. 
We begin with a little background on standard projection methods. Recall that  
given a Hermitian $n \times n$ matrix $M$, a Rayleigh-Ritz (RR)  projection method  
extracts approximate eigenpairs for $A$ from a \textit{search subspace} $\Span \{ Z \}$, 
where $Z \ \in \IR^{n \times s}$. It does so by imposing two 
conditions. First, any approximate eigenvector must belong to $\Span \{ Z \}$, 
i.e., it can be written as $z = Z c$ (where 
$c \ \in \IR^{s}$). Second, the approximate eigenpair $(\theta, z)$ must 
satisfy the Galerkin condition that  
the residual is orthogonal to $\Span \{ Z \}$, i.e., we must have 
$(M - \theta I) z  \perp \Span \{ Z \}$. This yields the projected 
$s \times s$ eigenvalue problem $(Z^T M Z) c = \theta c$ from which  
we obtain the Ritz values $\theta_i$, the corresponding eigenvectors $c_i$,  
and the Ritz vectors 
$z_i = Z c_i, i=1,\ldots, s$.  Details 
can be found in, e.g., \cite{Parlett-book,Saad-book3}.  
 
\subsection{Application to the SVD}\label{sec:RRSVD} 
Consider now the singular value problem for a matrix $A \in \ \IR^{m\times n}$.  
The above projection procedure  can be adapted to the singular value problem in a number of 
ways. For example, we can apply the RR idea to one of the 
two standard eigenvalue problems with either  $A^T A$ or  $A A^T$. 
This, however, is not appealing due to its ``non-symmetric'' nature: it puts an emphasis 
on one set of singular vectors (left or right) and will encounter difficulties with the  
smallest singular values due to their squaring. An alternative  
approach is to apply the Rayleigh-Ritz procedure to the  augmented matrix 
\eq{eq:augment}  
B = \begin{pmatrix}  
 0 & A \\ A^T & 0   
\end{pmatrix}  .  
\en 
We will consider this second approach as it leads more naturally to a separation of the  
right and left singular vectors. A straightforward application of the RR procedure would 
now use  a subspace $\Span \{ Z \}$ where  
$Z \in \ \IR^{(m+n) \times s}$ and would write a test eigenvector in the form  
$z = Z c$ where $c \ \in \IR^{s}$. It then imposes the Galerkin condition  
$Z^T(B - \theta I) Z c = 0$ from which Ritz values and vectors are  
obtained. 
Observe that $z$ can be written as $z=\left(u \atop v\right)$,  
where $u \ \in \IR^m$ approximates a left singular vector of $A$ and  
$v \ \in \IR^n$ approximates a right singular vector of $A$. 
 
One weakness of 
this viewpoint is that there is no reason why 
the left approximate singular vector (vector $u$, i.e, top part of $z=Zc$) and  
the right approximate singular vector (vector $v$ or bottom part of $z=Zc$) should be expressed in 
the basis $Z$ with the same  basis coefficients $c$. In practice, we have 
two  bases available, one for the left singular vectors  
and one for the right singular vectors. Therefore, let   
$U \ \in \IR^{m \times s_1}$ be a basis for the \textit{left search subspace} and 
$V \ \in \IR^{n \times s_2}$ a basis for the \textit{right search subspace}, 
with $s_1+s_2=s$. 
Both $U$ and $V$ are assumed to be orthonormal bases and note that  
$s_1$ and $s_2$ need not be the same. 
Then, the approximate right singular vector can be expressed as $u = U f $ (with $f\ \in \IR^{s_1}$)  
and the approximate left singular vector as $v = V g $ (with $g\ \in \IR^{s_2}$).  
This gives us $s_1+s_2=s$ degrees of freedom. To extract $f$ and $g$ we would need 
$s $ constraints which are to be imposed on the residual vector $(B - \theta I) z$  
where $z=\left(u \atop v\right)$. This residual is  
\eq{eq:res} r = \begin{pmatrix}  
 A v - \theta u \\ A^T u - \theta v  
\end{pmatrix} . \en 
It is natural to impose the condition that the first part,  
which is in $\IR^m$, be orthogonal to $\Span \{ U \} $  
and the second  part, which is in $\IR^n$, be orthogonal to $\Span \{ V \} $ : 
\eq{eq:Gal2} 
\left\{ \begin{array}{ccc}  
 U^T (A V g - \theta U f )  & = & 0  \\  
 V^T (A^T U f - \theta V g )   & = & 0  
\end{array} \right. .  
\en 
System~(\ref{eq:Gal2}) leads to the projected singular value  
problem $H g = \theta f$ and $H^T f = \theta g$, where $H = U^T A V$. 
Let $(\theta_i, f_i, g_i)$ be the singular triplets of the projected matrix~$H$,  
i.e., $H g_i = \theta_i f_i$ and $H^T f_i = \theta_i g_i$; $i=1,\ldots, \min\{s_1,s_2\}$. 
Then the scalars $\theta_i$ are the \textit{Ritz singular values}, the vectors $U f_i$ are the  
\textit{left Ritz singular vectors} of $A$, and the vectors $V g_i$ are the  
\textit{right  Ritz singular vectors} of $A$. 
It is important to note that when $\theta_i >0$ then 
 the above conditions imply that $ \|f_i \| = \| g_i \|$. 
This can be easily seen by multiplying (from left) both sides of the equality  
$H g_i = \theta_i f_i$ by $f_i^T$ and of the equality $H^T f_i = \theta_i g_i$ by $g_i^T$. 

The  ``doubled'' form  of the  Galerkin  condition~(\ref{eq:Gal2}) has 
been  considered  in  \cite{Hochstenbach01}  in  the  context  of  the 
correction  equation for  a  Jacobi-Davidson approach.  These are  not 
quite standard  Galerkin conditions since they amount  to two separate 
orthogonality  constraints.   
However,  it  is also possible  to interpret this approach  
as  a  standard 
Galerkin/RR procedure, which is the viewpoint we  
develop next.  
 
Consider the following 
new basis  for a subspace of  $\IR^{m+n}$ given by   
\eq{eq:Zbasis} Z = 
\left[ 
\begin{matrix}  
U & 0  \\ 
0 & V \end{matrix} \right] . \en 
Then a test vector $z$ in $\Span \{ Z \}$ can be written as  
\eq{eq:vecz} 
z = \left[ \begin{matrix}  
U & 0  \\ 
0 & V \end{matrix} \right]  \begin{pmatrix} f \\ g \end{pmatrix} =  \begin{pmatrix} u \\ v \end{pmatrix}  .  
\en  
Clearly, the residual vector $(B - \theta I)z$ of this test vector for  
the approximate eigenvalue $\theta $ and the matrix $B$ is again given by \eqref{eq:res} and the standard Galerkin condition 
$Z^T r = 0$ yields exactly the doubled form \eqref{eq:Gal2} of the Galerkin condition. 
\begin{proposition} 
The RR procedure defined by  
the doubled form of the Galerkin condition \eqref{eq:Gal2} is mathematically equivalent to the 
standard RR procedure applied to the symmetric matrix $B$ using the basis  
given by \eqref{eq:Zbasis}.  
\end{proposition}  
 
In essence the procedure defined in this way puts an emphasis on \textit{not mixing} the  
$U$-space and the $V$-space as 
indicated by zeros in the appropriate locations in \eqref{eq:Zbasis} and this is  
achieved by a restricting the  choice of basis for the search space. 
This is in contrast to the 
general procedure described at the very beginning of this section, where $Z$ makes  
no distinction between the $U$- and $V$- spaces.  
 
\subsection{Application to LSI}  
Let $A$ be a certain  term-document matrix considered at some stage  
in the updating and querying process and let 
$U \in \mathbb{R}^{m \times s_1}$ and $V \in \mathbb{R}^{n \times s_2}$  
be the orthonormal bases of the left and right search subspaces, respectively. 
Suppose that our goal is to construct approximations $(\tilde \sigma_i, \tilde u_i, \tilde v_i)$  
to the $k$ dominant singular triplets $(\sigma_i, u_i, v_i)$, such that 
$\tilde \sigma_i   = \theta_i \geq 0$, $\tilde u_i = U f_i$, and  
$\tilde v_i = V g_i$; $i = 1, \ldots k$.  
The matrices of the singular values and vectors of $A$ are then approximated by  
\begin{equation}\label{eqn:approx_form} 
\tilde \Sigma_k = \Theta_k, \quad \tilde U_k =  U F_k, \quad  \tilde V_k =  V G_k \; ; 
\end{equation} 
where $\Theta_k = \mbox{diag} \left\{ \theta_1, \ldots , \theta_k \right\}$, and 
$F_k = [f_1, \ldots, f_k] \in \mathbb{R}^{s_1 \times k}$ 
and $G_k = [g_1, \ldots, g_k] \in \mathbb{R}^{s_2 \times k}$ are the ``coefficient matrices''  
with orthonormal columns. 
The resulting $\tilde \Sigma_k$, $\tilde U_k = [\tilde u_1, \ldots, \tilde u_k]$,  
and $\tilde V_k = [\tilde v_1, \dots, \tilde v_k]$ can be used  
to evaluate the relevance scores in~(\ref{eqn:scores}) 
instead of the exact $\Sigma_k$, $U_k$, and $V_k$.   
 
A solution to this problem can be obtained by simultaneously imposing the  
\textit{Galerkin} conditions seen in Section~\ref{sec:RRSVD} to the $k$ residuals, which  
along with the assumption on $(\tilde \sigma_i, \tilde u_i, \tilde v_i)$ 
lead to the equations   
\begin{equation}\label{eqn:proj_prob} 
\left\{ 
\begin{array}{lcl} 
(U^T A  V) g_i & = & \theta_i f_i  \\  
(U^T A  V )^T f_i & = & \theta_i g_i   
\end{array}  
\right., \qquad i = 1,\ldots,k \; . 
\end{equation} 
The unknown triplets $ (\theta_i, f_i, g_i)$ are  
determined by the SVD of the projected matrix  
$H =  U^T A  V \in \mathbb{R}^{s_1 \times s_2}$.  
The approximations to the $k$ dominant singular triplets of $A$ are then  
defined  
by~(\ref{eqn:approx_form})  
where $\Theta_k$, $F_k$, and $G_k$ correspond to   
the $k$ dominant singular triplets of $H$.  
The  diagonal entries of $\Theta_k$ are the Ritz singular values and  
the columns of $\tilde U_k$ and $\tilde V_k$ in~(\ref{eqn:approx_form}) are  
the left and right  Ritz singular vectors, respectively.

We will refer to the above approximation scheme as  
the \textit{singular value Rayleigh-Ritz} procedure  
for the matrix $A$ with respect to $U$ and $V$ or, shortly, SV-RR($A$, $U$, $V$).  
The overall scheme is summarized in Algorithm~\ref{alg:sv-rr}. 
 
\begin{algorithm}[SV-RR ($A$, $U$, $V$)]\label{alg:sv-rr} 
\emph{Input:} $A$, $U$, $V$. \emph{Output:} $\tilde \Sigma_k$, $\tilde U_k$, $\tilde V_k$. 
\begin{enumerate}  
\item \emph{Construct} $H = U^T A  V$.  
\item \emph{Compute the SVD of} $H$. \emph{Form matrices}  
$\Theta_k$,  
$F_k$, \emph{and} $G_k$ \emph{that correspond to the $k$ dominant singular triplets of} $H$.   
\item \emph{Return} $\tilde \Sigma_k$, $\tilde U_k$, \emph{and} $\tilde V_k$, 
\emph{given by} (\ref{eqn:approx_form}).  
 
\end{enumerate} 
\end{algorithm} 
\vspace{0.1in} 

\subsection{Optimality} 
Since the process just described is a standard RR 
procedure applied to the matrix \eqref{eq:augment}, well-established optimality  
results for the RR procedure apply. Here, we show a few consequences 
of this viewpoint, some of which coincide with results that can be 
found elsewhere, see, e.g., \cite{Hochstenbach01}, using a different approach.  
The presented Min-Max results, however, are new to the best of our knowledge.
 
Let us first examine the Rayleigh Quotient (RQ) associated 
with the augmented matrix $B$ for a vector 
$z = \left(u \atop v\right)$. We find that  
\eq{eq:RQB} 
\rho (z) \equiv \frac{ (B z,z) }{(z,z)}  =\frac{2 u^T A v }{ \| u\|^2 + \| v\|_2^2}.  
\en
As a next step we study properties of this function at points that yield
its maximum values in subspaces of $\Span(Z)$, where $Z$ is defined in~\eqref{eq:Zbasis}.
The results will be used shortly to establish an optimality result 
for the SV-RR procedure.

We start with by observing that any subspace $\calS$ of $\Span(Z)$ has a special structure:
\eq{eq:SStruct}
\calS = \left\{ \left( u \atop v \right) \in \IR^{m+n} : u \in \calS_u \ \text{and} \ v \in \calS_v \right\},
\quad \dim(\calS) = \dim(\calS_u) + \dim(\calS_v),
\en
where $\calS_u$ and $\calS_v$ are some subspaces of $\mathcal{U} = \Span(U)$ and 
$\mathcal{V} = \Span(V)$, respectively.
In particular, representation~\eqref{eq:SStruct} suggests that if 
$z = \left( u \atop v \right)$ is in $\calS$ then, for any scalars $\alpha$ and $\beta$,  
the vector $\bar z = \left( \alpha u \atop \beta v \right)$ is also in $\calS$.
The following lemma is a simple consequence of this property.

\begin{lemma}\label{lem:OptProp1}  
The maximum of the RQ in~\eqref{eq:RQB} over a subspace $\calS \subseteq \text{\emph{span}}(Z)$ is 
nonnegative.
\end{lemma}
\begin{proof} 
Let $z^* = (u_*^T \ v_*^T)^T$ be a maximizer of~\eqref{eq:RQB} in $\Span(Z)$, 
and assume that the maximum is negative, i.e., $\rho(z_*) < 0$. 
The vector $\bar z = \left( -u_*^T \  v_*^T \right)^T$
is also in $\calS$, and we have $\rho( \bar z ) = - \rho(z_*) > 0 > \rho(z_*)$,
contradicting our assumption that $\rho(z_*)$ is the maximum.
Therefore,  $\rho(z_*) \ge 0$. 
\end{proof}

As seen earlier, if $\theta_i >0$ then the conditions  
$H g = \theta f $ and $H^T f = \theta g$, arising in the projection procedure, imply that  
$\| g \| = \| f \| $, so that the approximate singular vectors 
satisfy $\|u\| = \|v\|$.  
Viewed from a different angle, we can now ask the question:
Does the 
maximizer of the RQ \eqref{eq:RQB} over $z \ \in \ \Span\{Z\}$, 
satisfy the property that $\| u \| = \|v\|$? 
The answer to the question is yes as the following lemma shows.

\begin{lemma}\label{lem:OptProp2}  
Let the maximum of the RQ \eqref{eq:RQB} over a subspace 
$\mathcal{S} \subseteq \text{\emph{span}}(Z)$ be achieved at a vector 
$z_* = \begin{pmatrix}  u_* \cr v_*\end{pmatrix}$.
If the maximum is positive then~$\|u_*\| = \|v_*\|$. 
\end{lemma} 
\begin{proof} 
Let us assume that the  contrary is true, i.e., that $\|u_*\| \ne \|v_*\|$. 
Since $\rho(z_*) > 0$, both $u_*$ and $v_*$ are nonzero. 
Thus, we can define $\alpha =\sqrt{ \| v_* \| / \| u_* \|}$ and introduce
the vector $\bar z = \left( \alpha u_*^T \ (1 / \alpha) v_*^T \right)^T$,
which is also in $\calS$. But then $\rho(\bar z) = (2 u_*^T A v_*) / (2 \|u_*\| \|v_*\|)$ 
which is larger than $  \rho(z_*)$ under the assumption that $\|u_*\| \ne \|v_*\|$. 
Hence we have increased the value of $\rho(z_*)$ contradicting  
the fact that it is the maximum. Therefore we must have $\| u_* \| = \| v_* \|$. 
\end{proof}  

While the above result concerns the case where the maximum of the RQ is positive,
the next lemma shows that if a subspace $\calS$ is sufficiently large
then one can also find vectors $z \in \calS$ with $\|u\| = \|v\|$ 
that correspond to a zero maximum.    

\begin{lemma}\label{lem:OptProp3}  
Let the maximum of the RQ~\eqref{eq:RQB} over $\calS \subseteq \text{\emph{span}}(Z)$ be zero and assume that
$\dim(\calS) > \max(s_1, s_2)$, where $s_1 = \dim(\calU)$ and $s_2 = \dim(\calV)$. 
Then there exists a maximizer $z^* \in \calS$, such that $\|u^*\| = \|v^*\|$. 

\end{lemma} 
\begin{proof} 
 First we show the existence of a  maximizer $z_0$ of~\eqref{eq:RQB} with nonzero 
components $u_0$ and $v_0$. 
The condition  $\dim(\calS) \equiv \dim (\calS_u) + \dim (\calS_v) > \max(s_1, s_2)$,
with $ \dim (\calS_u) \le s_1$ and $ \dim (\calS_v) \le s_2$, implies that 
neither of the subspaces $\calS_u, \calS_v$ has zero dimension, so 
$\calS$ contains a vector $z_0 = (u_0^T, v_0^T)^T$ such that  $u_0, v_0 \ne 0$.
Without loss of generality, we assume that $\rho(z_0) \geq 0$ (otherwise, choose
$z_0 = (-u_0^T, v_0^T)^T$ which is also in $\calS$). But
we cannot have $\rho(z_0) >0$ since the maximum of $\rho$ is zero.
Therefore $\rho(z_0) = 0$ and $z_0$ is the desired vector.
To complete the proof, define $z_* \equiv (u_*^T, v_*^T)^T$ with
$u_* = u_0 /\| u_0 \|$, $v_* = v_0 /\| v_0 \|$. We clearly still have $\rho(z_*)=0$ 
and the vector $z_*$, which belongs to $\calS$, has the desired property.
\end{proof}


The lemmas show that under the specified conditions, 
the vectors $u$ and $v$ of the RQ maximizers are (or can be chosen to be) of 
the same norm which can be set to one without loss  of generality. 
In particular, since $\theta_1 $ is the maximum of the RQ over the  
space of all nonzero vectors $z$ of the form \eqref{eq:vecz}, we readily 
obtain the following  
characterization of the largest Ritz singular value,
which can also be found in~\cite{Hochstenbach01}:  
\eq{eq:theta1} 
\theta_1 = \max_{u \in \calU; v \in \calV; \atop \| u \| = \| v \|=1} \quad  u^T A v \geq 0.
\en 
 
We now  wish to generalize this result by 
establishing  a min-max type characterization for all Ritz singular values. 
For an $n\times  n$ Hermitian  matrix $M$ and a search subspace $\calZ$
of dimension $s$ the decreasingly labeled Ritz values $\theta_i$ 
are given by
\eq{eq:minmaxP1}  
\theta_i = 
\min_{\calS \subseteq \calZ   \atop \text{dim}  
(\calS) = s-i+1 }  \quad \max_{x \ \in \ \calS ,  x \ne 0} \quad  
\frac{(Mx,x)}{(x,x)}, \quad i = 1, \ldots, s;   
\en
see, e.g.,~\cite{Parlett-book,Saad-book3}.
This expression, combined with the results of the current section, applied to our situation, 
where $M = B$ is the augmented matrix in  
\eqref{eq:augment} and $\calZ = \Span\{Z\}$ where  $Z$ is defined 
in \eqref{eq:Zbasis}, results in the following  
Min-Max characterization of the Ritz singular values. 
\begin{theorem}\label{thm:sv_charact} 
The Ritz singular values of a matrix $A \in \IR^{m \times n}$     
with respect to the left and right search subspaces  
$\calU = \text{\emph{span}}\{U\} \subseteq \IR^{m}$ and  
$\calV = \text{\emph{span}}\{V\}\subseteq \IR^{n}$,
 with $\text{dim} (\calU ) = s_1$, $\dim (\calV )= s_2$,  and $s_1+s_2=s$,
admit the following characterization: 
\eq{eq:minmaxP2} 
\theta_i =  
\min_{\calS_u \subseteq \calU,  \calS_v \subseteq \calV, 
\atop \text{dim} (\calS_u) + \text{dim} (\calS_v) = s-i+1}  
\quad  
\max_{  
u \in \calS_u  ; \ v \in \calS_v ; 
 \atop  
 \| u \| = \| v\|=1   
} \quad  u^T A v,  \qquad i = 1, \ldots, \min(s_1, s_2) . 
\en  
\end{theorem} 
 
\begin{proof}  
We start from \eqref{eq:minmaxP1}. 
Let  $\calZ$ be spanned by a basis of the form \eqref{eq:Zbasis}. 
A candidate   subspace $\calS$  of $\calZ$ of this type is of the 
form~\eqref{eq:SStruct}.
The dimension of this 
subspace $\calS$  must be $s-i+1$, which translates to the  
requirement  that $\text{dim}\ (\calS_u) +  
\text{dim}\ ( \calS_v ) = s-i+1$. Next, 
we replace $M$ by  
$B$ in \eqref{eq:minmaxP1}. Then $(Mx,x)/(x,x) $ yields the expression in  
\eqref{eq:RQB}. 
Lemma~\ref{lem:OptProp2} shows that a positive 
maximum of this RQ in \eqref{eq:minmaxP1} is 
reached at vectors with $\|u\|=\|v\|$, so we can scale both $u$ and $v$ to have  
unit norm.
Note that the assumption $i \le \min(s_1, s_2) $ implies that 
$\dim(\calS)$ is larger than $\max(s_1, s_2)$. 
Therefore, if the maximum is zero then, by Lemma~\ref{lem:OptProp3},
there exists a maximizer with $\|u\| = \|v\|$ which can also be scaled so that
both $u$ and $v$ have unit norm. This yields \eqref{eq:minmaxP2}.  
\end{proof} 

Note that the same ideas applied to the augmented matrix~\eqref{eq:augment} 
with the search subspace $\Span(Z) = \IR^{m+n}$, where 
$U$ and $V$ in~\eqref{eq:Zbasis} are orthonormal bases of $\IR^m$ and $\IR^n$, 
respectively, lead to the Min-Max characterization of the singular values of $A$:
\eq{eq:minmaxSV} 
\sigma_i =  
\min_{\calS_u \subseteq \IR^m,  \calS_v \subseteq \IR^n, 
\atop \text{dim} (\calS_u) + \text{dim} (\calS_v) = m+n-i+1}  
\quad  
\max_{  
u \in \calS_u  ; \ v \in \calS_v ; 
 \atop  
 \| u \| = \| v\|=1   
} \quad  u^T A v,  \qquad i = 1, \ldots, \min(m, n) . 
\en  
One can see that~\eqref{eq:minmaxP2}  
is identical to~\eqref{eq:minmaxSV} with $\IR^m$ replaced by $\calU$
and $\IR^n$ by $\calV$.  This observation provides an interpretation of the 
optimality of the SV-RR procedure:  it constructs a solution
that admits the same characterization of the singular values in the given subspaces.   

The following statement is a direct consequence of Theorem~\ref{thm:sv_charact}. 
\begin{corollary}\label{cor:monoton} 
Let $\sigma_i$ and $\theta_i$ be labeled in a decreasing order. 
Then the Ritz values approximate the singular values from below, i.e., 
$\theta_i \le \sigma_i$; $i = 1, \ldots, \min(s_1, s_2)$.  
Additionally, if we have  a sequence of  
expanding subspaces $\left\{ \calZ_j \right\}$, such that   
$\calZ_j \subseteq \calZ_{j+1}$, 
then $ \theta_i^{(j)} \le \theta_i^{(j+1)} \le \sigma_i $, where 
$\theta_i^{(j)}$ is the $i$-th Ritz singular value with respect to the
subspace $\calZ_j$. 
\end{corollary} 
 
 
 
 
The results of this section suggest that the proximity  
of $\tilde \Sigma_k$,  $\tilde U_k$, and  $\tilde V_k$, produced by Algorithm~\ref{alg:sv-rr}, 
to the singular  triplets of $A$ is governed by the choice 
of $U$ and $V$. If both are properly chosen, then  SV-RR($A$, $U$, $V$)  
can give an appealing approach for the  updating  problem.   
In particular, as can be seen from Theorem~\ref{thm:sv_charact},  
if $\Span (U)$ and $\Span(V)$ contain the left and  
right dominant singular subspaces of $A$, $\Span (U_k)$ and $\Span(V_k)$,  
then Algorithm~\ref{alg:sv-rr} readily delivers the $k$ \textit{exact} singular  
triplets of interest. 
If, additionally, the number of columns in $U$ or $V$ is sufficiently small 
(not much larger than $k$) then the computational costs related to the procedure
can be negligibly low.  
 
In  practice, however, the construction  of search  subspaces that 
include the  targeted singular subspaces  and, at the same  time, have 
small dimensions  can be problematic.  It  is likely that  in order to 
obtain the inclusion  of the singular subspaces, both  $s_1$ and $s_2$ 
have  to be  large (see Corollary~\ref{cor:monoton}),  
and  this can  make  Algorithm~\ref{alg:sv-rr} too 
costly to  be used  as a  fast updating scheme.   In what  follows, we 
adapt this  viewpoint to  address possible difficulties  with existing 
methods.


\section{LSI updating methods viewed as projection  
procedures}\label{sec:lsi_proj} 
We now consider a common situation in IR which  
arises when a certain term-document matrix $A$ is updated. In 
the case when a few documents are added,   
the updated term-document matrix can be written as  
\begin{equation}\label{eqn:update_doc} 
\tilde A_D = [A, \ D], 
\end{equation}  
where $D \in \mathbb{R}^{m \times p}$, 
represents the matrix of added documents.  
Similarly, if terms are added, then the update matrix 
takes the form  
\begin{equation}\label{eqn:update_term} 
\tilde A_T = \left[ 
\begin{array}{c} 
 A \\  
 T  
\end{array} 
\right], 
\end{equation}  
where $T \in \mathbb{R}^{p \times n}$ corresponds to the added terms.  

It is often the case that the arrival of new documents triggers the addition
of new terms. In this situation the overall update can be decomposed into the sequence
of the two basic ones:~\eqref{eqn:update_term} followed by~\eqref{eqn:update_doc}. 
Once the incorporation of the new terms and documents is completed, the weights
of several affected terms should be accordingly adjusted. 
This motivates yet another type of update.

In particular, it is frequently required
to correct the weights  of the selected $p$ terms
throughout the whole document collection, 
in which case  
the updated  $A$ is given by    
\begin{equation}\label{eqn:update_cw} 
\tilde A_{CW} = A + C W, 
\end{equation}  
where $C \in \mathbb{R}^{m \times p}$ is a ``selection matrix'' 
and $W \in \mathbb{R}^{p \times n}$ contains weight corrections.     
The rows of $C$ that correspond to the terms whose weights are corrected 
represent the rows of the $p$-by-$p$ identity matrix, whereas the remaining rows of 
$C$ are zero.  Each row of the matrix $W$ contains the differences between 
the old and new weights for the corresponding term, with different entries specifying 
corrections to different documents.

The  goal of  the  updating algorithms  is  to compute  approximations 
$\tilde  \Sigma_k$, $\tilde  U_k$, and  $\tilde V_k$  to  the dominant 
singular   triplets of the updated matrices in~(\ref{eqn:update_doc})--(\ref{eqn:update_cw}) 
by exploiting the  knowledge of the singular triplets  $\Sigma_k$, $U_k$, 
and $V_k$ of $A$. 
A common approach is based on the idea of replacing $A$  
by its best rank-$k$ approximation $A_k = U_k \Sigma_k V_k^T$, and  
considering 
\begin{equation}\label{eqn:updates_lowrank} 
A_D = [A_k \ D], \  
\quad  
A_T = \left[ 
\begin{array}{c} 
 A_k \\  
 T  
\end{array} 
\right],  
\quad  
\ \mbox{and} \  \quad  
A_{CW} = A_k + C W  
\end{equation} 
as  substitutes for the updated matrices        
in~(\ref{eqn:update_doc})--(\ref{eqn:update_cw}). 
The $k$ dominant singular triplets of~(\ref{eqn:updates_lowrank}) are then regarded as  
approximations of the ``true'' updated singular triplets of~(\ref{eqn:update_doc})--(\ref{eqn:update_cw}), 
and are used to evaluate the relevance scores in~(\ref{eqn:scores}).

The result of Zha and Simon~\cite{Zha.Simon:99} shows that it is possible to compute the \textit{exact} 
$k$ dominant singular triplets of the matrices in~(\ref{eqn:updates_lowrank})  
without invoking standard singular value solvers from scratch.  
The closeness of the computed triplets to those of $\tilde A_D$, $\tilde A_T$,  
and $\tilde A_{CW}$ in~(\ref{eqn:update_doc})--(\ref{eqn:update_cw}) 
is justified by exploiting the so-called approximate ``low-rank-plus-shift''  
structure of $A$; see also~\cite{Zha.Zhang:00} for a more rigorous analysis.  
Below, we briefly review the updating schemes presented in \cite{Zha.Simon:99}. 
 
\subsection{Updating algorithms of Zha and Simon~\cite{Zha.Simon:99}}\label{subsec:zha} 
 
Consider first the case of adding new documents $D$. 
Let  
\begin{equation}\label{eqn:qr_doc} 
(I - U_k U_k^T) D = \hat U_p R  
\end{equation} 
be the truncated QR decomposition of $(I - U_k U_k^T) D$, where 
$\hat U_p \in \mathbb{R}^{m \times p}$ has orthonormal columns and $R \in \mathbb{R}^{p \times p}$ is  
upper triangular.  
Given~(\ref{eqn:qr_doc}), one can observe that   
\begin{equation}\label{eqn:relation_docs} 
A_D = [U_k, \ \hat U_p] 
H_D 
\left[ 
\begin{array}{cc} 
 V_k & 0 \\  
 0   & I_p  
\end{array} 
\right]^T \;, \ 
H_D = \left[ 
\begin{array}{cc} 
 \Sigma_k & U_k^T D \\  
 0        & R  
\end{array} 
\right] \;, 
\end{equation} 
where $I_p$ denotes the $p$-by-$p$ identity matrix.  
Thus, if $\Theta_k$, $F_k$, and $G_k$ are the matrices corresponding to the $k$ dominant singular values of  
$H_D \in \mathbb{R}^{(k+p) \times (k+p)}$ 
and their left and right singular vectors, respectively,  
then the desired updates $\tilde \Sigma_k$, $\tilde U_k$, and $\tilde V_k$ are given by 
\begin{equation}\label{eqn:upd_docs} 
\tilde \Sigma_k = \Theta_k, \ \tilde U_k = [U_k, \ \hat U_p] F_k, \ \mbox{and} \ \tilde V_k =  
\left[ 
\begin{array}{cc} 
 V_k & 0 \\  
 0   & I_p  
\end{array} 
\right] G_k \;.  
\end{equation}    
This updating procedure is summarized in the following algorithm.

\begin{algorithm}[Adding documents (Zha--Simon~\cite{Zha.Simon:99})]\label{alg:doc_zha} 
\emph{Input:} $\Sigma_k$, $U_k$, $V_k$, $D$. \emph{Output:} $\tilde \Sigma_k$, $\tilde U_k$, $\tilde V_k$. 
\begin{enumerate}  
\item \emph{Construct} $(I - U_k U_k^T) D$. 
\emph{Compute the QR decomposition~(\ref{eqn:qr_doc})}. 
\item \emph{Construct} $H_D$ \emph{in}~(\ref{eqn:relation_docs}).  
\emph{Compute the matrices} $\Theta_k$, $F_k$, \emph{and} $G_k$ \emph{that correspond to the} $k$ \emph{dominant singular triplets of} 
$H_D$.   
\item  
\emph {Return} $\tilde \Sigma_k$, $\tilde U_k$, \emph{and} $\tilde V_k$  
\emph{defined by}~(\ref{eqn:upd_docs}).  
\end{enumerate} 
\end{algorithm} 
\vspace{0.1in}

Similarly, if  new terms are added, then the following equality holds: 
\begin{equation}\label{eqn:relation_terms} 
A_T =  
\left[ 
\begin{array}{cc} 
 U_k^T & 0 \\  
 0   & I_p  
\end{array} 
\right] 
H_T 
[V_k, \ \hat V_p] \;, \ 
H_T = \left[ 
\begin{array}{cc} 
 \Sigma_k & 0 \\  
 TV_k        & L  
\end{array} 
\right] \;. 
\end{equation} 
Here, $\hat V_p$ and $L^T$ are the factors in the truncated QR decomposition of $(I - V_k V_k^T) T^T$, 
\begin{equation}\label{eqn:qr_term} 
(I - V_k V_k^T) T^T = \hat V_p L^T \;. 
\end{equation} 
The updated singular triplets are then defined as  
\begin{equation}\label{eqn:upd_terms} 
\tilde \Sigma_k = \Theta_k, \ \tilde U_k =  
\left[ 
\begin{array}{cc} 
 U_k & 0 \\  
 0   & I_p  
\end{array} 
\right] F_k, \ \mbox{and} \  
\tilde V_k = [V_k, \ \hat V_p] G_k,   
\end{equation}    
where $\Theta_k$, $F_k$, and $G_k$ now denote the matrices of the $k$ dominant singular triplets of  
$H_T \in \mathbb{R}^{(k+p) \times (k+p)}$. 
 
\begin{algorithm}[Adding terms (Zha--Simon~\cite{Zha.Simon:99})]\label{alg:term_zha} 
\emph{Input:} $\Sigma_k$, $U_k$, $V_k$, $T$. \emph{Output:} $\tilde \Sigma_k$, $\tilde U_k$, $\tilde V_k$. 
\begin{enumerate}  
\item \emph{Construct} $(I - V_k V_k^T) T^T$. 
\emph{Compute the QR decomposition~(\ref{eqn:qr_term})}. 
\item \emph{Construct} $H_T$ \emph{in}~(\ref{eqn:relation_terms}).  
\emph{Compute the matrices} $\Theta_k$, $F_k$, \emph{and} $G_k$ \emph{that correspond to the} $k$  
\emph{dominant singular triplets of} $H_T$.   
\item  
\emph {Return} $\tilde \Sigma_k$, $\tilde U_k$, \emph{and} $\tilde V_k$   
\emph{defined by}~(\ref{eqn:upd_terms}).  
\end{enumerate} 
\end{algorithm} 
\vspace{0.1in}

Finally, in the case of correcting the term weights,   
\begin{equation}\label{eqn:relation_cw} 
A_{CW} =  
[U_k, \ \hat U_p]  
H_{CW} 
[V_k, \ \hat V_p]^T \;, \ 
H_{CW} =  
\left[ 
\begin{array}{cc} 
 \Sigma_k & 0 \\  
 0        & 0  
\end{array} 
\right] + 
\left[ 
\begin{array}{c} 
 U_k^T C  \\  
 R         
\end{array} 
\right]  
\left[  
W V_k, \ L  
\right] \;, 
\end{equation} 
where $\hat U_p$, $R$, $\hat V_p$, and $L$ are given by the truncated QR decompositions   
\begin{equation}\label{eqn:qr_cw} 
(I - U_k U_k^T) C = \hat U_p R, \qquad (I - V_k V_k^T) W^T = \hat V_p L^T.   
\end{equation} 
%
Thus, assuming that $\Theta_k$, $F_k$, and $G_k$ are the matrices of the $k$ dominant singular triplets  
of $H_{CW} \in \mathbb{R}^{(k+p) \times (k+p)}$, 
\begin{equation}\label{eqn:upd_cw} 
\tilde \Sigma_k = \Theta_k, \ \tilde U_k = [U_k, \ \hat U_p] F_k, \ \mbox{and} \ \tilde V_k = [V_k, \hat V_p] G_k.  
\end{equation}    
 
\begin{algorithm}[Correcting weights (Zha--Simon~\cite{Zha.Simon:99})]\label{alg:cw_zha} 
\emph{Input:} $\Sigma_k$, $U_k$, $V_k$, $C$, $W$. \emph{Output:} $\tilde \Sigma_k$, $\tilde U_k$, $\tilde V_k$. 
\begin{enumerate}  
\item \emph{Construct} $(I - U_k U_k^T) C$ \emph{and} $(I - V_k V_k^T) W^T$. 
\emph{Compute the QR decompositions}~(\ref{eqn:qr_cw}). 
\item \emph{Construct} $H_{CW}$ \emph{in}~(\ref{eqn:relation_cw}).  
\emph{Compute the matrices} $\Theta_k$, $F_k$, \emph{and} $G_k$ \emph{that correspond to the} $k$  
\emph{dominant singular triplets of} $H_{CW}$.   
\item  
\emph {Return} $\tilde \Sigma_k$, $\tilde U_k$, \emph{and} $\tilde V_k$  
\emph{defined by} (\ref{eqn:upd_cw}).  
\end{enumerate} 
\end{algorithm} 
\vspace{0.1in} 
 

\subsection{The Rayleigh-Ritz viewpoint}\label{subsec:RRview} 
It is easy to see that Algorithm~\ref{alg:doc_zha}  
is equivalent to SV-RR($A_D$, $U$, $V$) with  
\begin{equation}\label{eqn:MUV_doc} 
U = \left[ U_k, \ \hat U_p \right], \   
V =  
\left[ 
\begin{array}{cc} 
 V_k & 0 \\  
 0   & I_p  
\end{array} 
\right].  
\end{equation}   
Note that, in this case, the matrix $H_D$ in~(\ref{eqn:relation_docs}) is precisely the projected matrix $H$  
in step 1 of Algorithm~\ref{alg:sv-rr}.  
 
The fact that Algorithm~\ref{alg:sv-rr} yields the exact dominant singular   
triplets of $A_D$ comes from the observation that  
\[ 
\Range  ( A_D ) = \Span ( U_k )  \oplus \Range ((I - U_k U_k^T) D)  \;,  
\] 
and that the columns of $\hat U_p$ in~(\ref{eqn:qr_doc}) 
form an orthonormal basis of  
$\Range  ((I - U_k U_k^T) D)$. Hence, 
$\Span  (U) = \Range (A_D)$, 
i.e., the search subspace $\Span (U)$ with $U$ defined in~(\ref{eqn:MUV_doc})  
must contain the left dominant singular subspace of $A_D$.  
The corresponding right singular subspace is then  
contained in $\Range ( A_D^T  U )$. But 
\[ 
A_D^T U  =   
\left[ 
\begin{array}{cc} 
 A_k^T U_k & A_k^T \hat U_p \\  
 D^T U_k   & D^T \hat U_p   
\end{array} 
\right]  =  
\left[ 
\begin{array}{cc} 
 V_k \Sigma_k & 0 \\  
 D^T U_k   & D^T \hat U_p   
\end{array} 
\right]  =  
\left[ 
\begin{array}{cc} 
 V_k & 0 \\  
 0   & I_p  
\end{array} 
\right] 
\left[ 
\begin{array}{cc} 
 \Sigma_k & 0 \\  
 D^T U_k   & D^T \hat U_p   
\end{array} 
\right] ,  
\] 
implying that 
$ \Span (V) = \Range ( A_D^T U ) $.  
Thus, the search subspace $\Span(V)$ defined in~(\ref{eqn:MUV_doc})  
must contain the right singular subspace of $A_D$. 
As a result, since both $\Span (U)$ and $\Span (V)$  
contain the dominant singular subspaces, 
a run of SV-RR($A_D$, $U$, $V$) in Algorithm~\ref{alg:doc_zha} is indeed  
guaranteed to deliver the $k$ exact dominant singular triplets of $A_D$. 
 
Similarly, for  the case of added new terms,  
Algorithm~\ref{alg:term_zha} , can be interpreted  
as SV-RR($A_T$, $U$, $V$) with 
\begin{equation}\label{eqn:MUV_term} 
U =  
\left[ 
\begin{array}{cc} 
 U_k & 0 \\  
 0   & I_p  
\end{array} 
\right], \   
V = \left[ V_k, \ \hat V_p \right] \; ,  
\end{equation}   
where $\hat V_p$ is defined by the QR decomposition~(\ref{eqn:qr_term}). 
Algorithm~\ref{alg:cw_zha}  
is equivalent to SV-RR($A_{CW}$, $U$, $V$) with 
\begin{equation}\label{eqn:MUV_cw} 
U = \left[ U_k, \ \hat U_p \right], \ 
V = \left[ V_k, \ \hat V_p \right] \; ,  
\end{equation}   
where $\hat U_p$ and $\hat V_p$ are given by~(\ref{eqn:qr_cw}).  
The matrices $H_T$ and $H_{CW}$ in~(\ref{eqn:relation_terms}) and~(\ref{eqn:relation_cw}) 
are the projected matrices that arise at step 1 of Algorithm~\ref{alg:sv-rr} with the  
appropriate input. 
The fact that the dominant singular triplets of $A_T$ and $A_{CW}$ are computed \textit{exactly}  
can be deduced by following the arguments similar to those used above to justify the exactness 
of Algorithm~\ref{alg:doc_zha}.  
 
 Finally, note that the updating methods from~\cite{BeDuBri:95, 
  OBrien:94},  which preceded  the  schemes   of  Zha  and 
Simon~\cite{Zha.Simon:99},  can  also  be  easily  interpreted  in  the 
framework  of the  projection procedure  in Algorithm~\ref{alg:sv-rr}. 
For  example, the  case of  adding new  documents in~\cite{BeDuBri:95, 
  OBrien:94}  is  realized  by   SV-RR($A_D$,  $U_k$,  $V$)  with  $V$ 
in~(\ref{eqn:MUV_doc}). The updating  strategies for the remaining two 
types  of  update correspond  to  SV-RR($A_T$,  $U$,  $V_k$) with  $U$ 
in~(\ref{eqn:MUV_term}) and SV-RR($A_{CW}$, $U_k$, $V_k$).

\section{Updating by the SV-RR with smaller subspaces}\label{sec:new_updates} 
Consider     the    computational     complexity of 
Algorithm~\ref{alg:doc_zha}.   In   its  first  step,   the  algorithm 
performs $\mathcal{O}(mkp)$  operations to form  the matrix $(I  - U_k 
U_k^T)  D$ and  $\mathcal{O}(m  p^2)$ operations  to  complete the  QR 
decomposition~(\ref{eqn:qr_doc}).  Assuming that  $U_k^T  D$ has  been 
precomputed   at  step   1,   the   cost  of  the second     step   is 
$\mathcal{O}((k+p)^3)$, which corresponds to the cost of the SVD of 
the $(k+p)$-by-$(k+p)$  matrix $H_D$. Finally, in the  third step, the 
algorithm  requires   $\mathcal{O}(k^2(m+n)  +  mkp)$   operations  to 
evaluate~(\ref{eqn:upd_docs}).  The  complexities of all the updating 
algorithms considered in this paper  are summarized  in   
Tables~\ref{tbl:doc_complex}-\ref{tbl:cw_complex} of  
subsection~\ref{subsec:complexity}. 
 
The above analysis makes it clear   
that the complexity of Algorithm~\ref{alg:doc_zha} scales cubically with respect  
to the update size $p$. When $p$ is small this cubic scaling 
 behavior will have an un-noticeable effect.  
However, it can lead to a substantial slow down for moderate to large $p$. 
In other words, for sufficiently large updates, the  
performance of Algorithm~\ref{alg:doc_zha} 
can be dominated  by  
the SVD of the projected matrix in step 2 and, to a lesser extent,  
by the QR decomposition  
of $(I - U_k U_k^T) D$ in step 1.  
  
In subsection~\ref{subsec:RRview}, we have established the relation between  
the existing updating methods~\cite{Zha.Simon:99} and a projection scheme  
for the singular value problem, which allows us to interpret Algorithm~\ref{alg:doc_zha} as  
SV-RR($A_D$, $U$, $V$) with $U$ and $V$ specified in~(\ref{eqn:MUV_doc}).  
In particular, this finding suggests that the size of the  
potentially critical SVD in step 2 of Algorithm~\ref{alg:doc_zha} is determined by the  
dimensions of the left and right search subspaces. 
 
In this paper, we propose to reduce the dimension of at least one of the two 
search subspaces and perform the projection procedure with respect to the resulting  
smaller subspace(s). For example, in the case of adding new documents, 
we suggest to reduce the dimension of the \textit{left} search subspace $\Span (U)$. 
Based on different options for the dimension reduction,     
we devise new updating schemes that correspond to Algorithm~\ref{alg:sv-rr} 
with the ``reduced'' left search subspaces, $A \equiv A_D$, and $V$ in~(\ref{eqn:MUV_doc}).    
%

%
More precisely, the idea is to replace $U \in \mathbb{R}^{m \times (k+p)}$ in~(\ref{eqn:MUV_doc}) by a matrix  
\begin{equation}\label{eqn:new_U_doc} 
\bar U =  \left[ U_k, \ Z_l \right] \in \mathbb{R}^{m \times (k+l)} , \qquad  
Z_l \in \mathbb{R}^{m \times l}, \qquad l \ll  p,  
\end{equation} 
with a significantly smaller number of columns. 
The matrix $Z_l \in \mathbb{R}^{m \times l}$, to be determined later, is assumed to have orthonormal columns and is  
such that 
\begin{equation}\label{eqn:compl_U_doc} 
\Span (Z_l) \subset \Range \left( (I - U_k U_k^T) D \right) \; . 
\end{equation} 
This implies that 
$Z_l^T U_k = \mathbf{0}$  
and, hence, all the columns of $\bar U$ in~(\ref{eqn:new_U_doc}) are orthonormal. 
Thus, given~(\ref{eqn:new_U_doc}) and (\ref{eqn:compl_U_doc}),  
one can obtain an updating scheme by applying SV-RR($A_D$, $\bar U$, $V$)  
with $V$ defined in~(\ref{eqn:MUV_doc}). 
 Different choices of $Z_l$ will lead to different updating schemes. 
 
The following proposition states the general form of the projected matrix produced 
by~SV-RR($A_D$,~$\bar U$,~$V$).  
 
 
\begin{proposition}\label{prop:H_Z} 
The application SV-RR($A_D$, $\bar U$, $V$) 
of Algorithm~\ref{alg:sv-rr} with $\bar U$ defined  
in~(\ref{eqn:new_U_doc})--(\ref{eqn:compl_U_doc}) and  
$V$ in~(\ref{eqn:MUV_doc}) produces the $(k+l) \times (k+p)$ projected matrix 
\begin{equation}\label{eqn:H_Z} 
H = \left[ 
\begin{array}{cc} 
\Sigma_k    &    U_k^T D \\ 
0           &    Z_l^T (I - U_k U_k^T) D        
\end{array} 
\right] \; . 
\end{equation} 
\end{proposition} 
 
\begin{proof}  
The statement is verified directly by constructing $H = \bar{U}^T A_D V$.  
Since $A_k = U_k \Sigma_k V_k^T$,  
\[ 
A_D V = \left[A_k, \ D \right]  
\left[ 
\begin{array}{cc} 
 V_k & 0 \\  
 0   & I_p  
\end{array} 
\right] =  
\left[ U_k \Sigma_k, \ D \right] \; . 
\] 
Thus, using~(\ref{eqn:new_U_doc}),  
\[ 
H = \bar U^T (A_D V) = \left[ U_k, \ Z_l \right]^T \left[ U_k \Sigma_k, \ D \right] = 
\left[ 
\begin{array}{cc} 
 \Sigma_k & U_k^T D \\  
 Z^T_l U_k \Sigma_k  & Z^T_l D  
\end{array} 
\right] \;.  
\] 
Since, by~(\ref{eqn:compl_U_doc}), $Z_l^T U_k = \textbf{0}$,  
the $(2,1)$-block of $H$ is zero.  
From~(\ref{eqn:compl_U_doc}), we also note that $Z_l = (I - U_k U_k^T) Z_l$.  
Hence, the $(2,2)$-block equals $Z_l^T (I - U_k U_k^T) D$, 
which completes the proof.  
\end{proof}  
 
A similar idea can be applied for the case of adding new terms. In particular, we 
replace $V \in \mathbb{R}^{n \times (k+p)}$ in~(\ref{eqn:MUV_term}) by an $n$-by-$(k+l)$ matrix  
\begin{equation}\label{eqn:new_V_term} 
\bar V = [V_k, \ Z_l] \in \mathbb{R}^{n \times (k+l)}, \qquad Z_l \in \mathbb{R}^{n \times l}, \qquad l \ll  p, 
\end{equation}  
where $Z_l$ has orthonormal columns and   
\begin{equation}\label{eqn:compl_V_term} 
\Span (Z_l) \subset \Range \left( (I - V_k V_k^T) T^T \right) \; . 
\end{equation} 
Then, by~(\ref{eqn:compl_V_term}),  
$Z_l^T V_k = \mathbf{0}$ 
and, therefore, all the columns of $\bar V$ in~(\ref{eqn:new_V_term}) are orthonormal. 
Thus, an updating scheme can be obtained by applying SV-RR($A_T$, $U$, $\bar V$). 
The corresponding projected matrix is given by the following proposition. 
 
\begin{proposition}\label{prop:H_Z_terms} 
The application of SV-RR($A_T$, $U$, $\bar V$)  
of Algorithm~\ref{alg:sv-rr} with $U$ defined in~(\ref{eqn:MUV_term}) 
and $\bar V$ in~(\ref{eqn:new_V_term})--(\ref{eqn:compl_V_term}) 
produces the $(k+p) \times (k+l)$ projected matrix 
\begin{equation}\label{eqn:H_Z_terms} 
H = \left[ 
\begin{array}{cc} 
\Sigma_k    &    0 \\ 
T V_k       &    T (I - V_k V_k^T) Z_l        
\end{array} 
\right] \; . 
\end{equation} 
\end{proposition}

Finally, if the term weights are corrected, in~(\ref{eqn:MUV_cw}), we substitute $U$ by  
\begin{equation}\label{eqn:new_U_cw} 
\bar U = [U_k, \ Z^{(1)}_{l_1}], \qquad Z^{(1)}_{l_1} \in \mathbb{R}^{m \times l_1}, \qquad l_1 \ll  p, 
\end{equation} 
and $ V $ by 
\begin{equation}\label{eqn:new_V_cw} 
\bar V = [V_k, \ Z^{(2)}_{l_2}], \qquad Z^{(2)}_{l_2} \in \mathbb{R}^{n \times l_2}, \qquad l_2 \ll  p \; . 
\end{equation} 
Similarly, $Z^{(1)}_{l_1} \in \mathbb{R}^{m \times l_1}$  
and $Z^{(2)}_{l_2} \in \mathbb{R}^{n \times l_2}$ are assumed to have orthonormal columns, and 
are such that 
\begin{equation}\label{eqn:compl_UV_cw} 
\Span (Z^{(1)}_{l_1}) \subset \Range \left( (I - U_k U_k^T)C \right), \qquad 
\Span (Z^{(2)}_{l_2}) \subset \Range \left( (I - V_k V_k^T)W^T \right) \; . 
\end{equation} 
The latter implies that  
$Z_{l_1}^{(1)T} U_k = \textbf{0}$ 
and $Z_{l_2}^{(2)T} V_k = \textbf{0}$, 
i.e., both $\bar U$ and $\bar V$ in~(\ref{eqn:new_U_cw})--(\ref{eqn:new_V_cw}) have 
orthonormal columns.   
As a result, an updating scheme can be given by SV-RR($A_{CW}$, $\bar U$, $\bar V$),  
with $\bar U$ and $\bar V$ defined in~(\ref{eqn:new_U_cw})--(\ref{eqn:compl_UV_cw}).  
 
\begin{proposition}\label{prop:H_Z_terms_cw} 
The SV-RR($A_{CW}$, $\bar U$, $\bar V$) run of Algorithm~\ref{alg:sv-rr} with  
$\bar U$ and $\bar V$ defined in~(\ref{eqn:new_U_cw})--(\ref{eqn:compl_UV_cw})  
produces the $(k+l_1) \times (k+l_2)$ projected matrix 
\begin{equation}\label{eqn:H_Z_terms_cw} 
H =  
\left[ 
\begin{array}{cc} 
 \Sigma_k & 0 \\  
 0        & 0  
\end{array} 
\right] + 
\left[ 
\begin{array}{c} 
 U_k^T C  \\  
 Z^{(1)T}_{l_1} (I - U_k U_k^T) C         
\end{array} 
\right]  
\left[  
W V_k, \ W (I - V_k V_k^T) Z^{(2)}_{l_2}  
\right] \; . 
\end{equation} 
\end{proposition}

It        is         important        to        emphasize        that, 
in contrast to~Algorithms~\ref{alg:doc_zha}--\ref{alg:cw_zha},    the
updating framework     introduced 
above, which uses  smaller search  subspaces, is  no 
longer  expected to  deliver  the exact  singular  triplets of  $A_D$, 
$A_T$,  or  $A_{CW}$.   However,   as  the  numerical  experiments  in 
section~\ref{sec:num}  will   demonstrate,   satisfactory  retrieval 
accuracies can be achieved without computing these triplets exactly, and 
only  approximations  suffice in  practice.  This  observation can  be 
related to the  fact that the matrices $A_D$,  $A_T$, and $A_{CW}$ are 
themselves  approximations of  the ``true''  updated  matrices $\tilde 
A_D$,       $\tilde        A_T$,       and       $\tilde       A_{CW}$ 
in~(\ref{eqn:update_doc})--(\ref{eqn:update_cw}), and so there is no 
need to compute their singular triplets with high accuracy.  
 
The rest of this section studies several choices of $Z_l$  
in~(\ref{eqn:new_U_doc})--(\ref{eqn:compl_U_doc}) and~(\ref{eqn:new_V_term})--(\ref{eqn:compl_V_term}),  
as well as of $Z^{(1)}_{l_1}$ and $Z^{(2)}_{l_2}$ in~(\ref{eqn:new_U_cw})--(\ref{eqn:compl_UV_cw}).

\subsection{Optimal rank-$l$ approximations.} 
 
We start with the approach based on SV-RR($A_D$, $\bar U$, $V$),  
where $\bar U$ is assumed to satisfy~(\ref{eqn:new_U_doc})--(\ref{eqn:compl_U_doc})  
and $V$ is defined in~(\ref{eqn:MUV_doc}).  
Our goal is to specify a suitable choice of $Z_l$ and turn the general projection procedure  
into a practical updating algorithm that handles the case of adding new documents.  
 
As has been pointed out in subsection~\ref{subsec:RRview},  
a run of SV-RR($A_D$, $U$, $V$) 
with $U$ and $V$ from~(\ref{eqn:MUV_doc}) is equivalent to Algorithm~\ref{alg:doc_zha}, which 
is known to compute the exact dominant singular triplets of $A_D$. This property of the algorithm,  
paired with the analysis in~\cite{Zha.Zhang:00}, explains the generally satisfactory retrieval  
quality maintained by the updating scheme. From these considerations, it is desirable that the new search 
subspace $\Span(\bar U)$ utilized by SV-RR($A_D$, $\bar U$, $V$) approximates $\Span( U )$.  
In this case, our expectation is that SV-RR($A_D$, $\bar U$, $V$)  
will exhibit a retrieval accuracy comparable to that 
of SV-RR($A_D$, $U$, $V$) implemented by  
Algorithm~\ref{alg:doc_zha}.  
 
By definition, $\Span (\bar U) = \Span (U_k) \oplus \Span (Z_l)$ 
and $\Span (U) = \Span (U_k) \oplus \Range \left( (I - U_k U_k^T)D \right)$.  
Therefore, we seek to construct $Z_l$ such that $\Span (Z_l)$ approximates  
$\Range ( (I - U_k U_k^T) D )$. In other words, we would like to  
compress the information contained in 
the subspace $\Range\left( (I - U_k U_k^T) D \right)$ into the span of 
a set of  $l$ orthonormal vectors.  
Formally, this requirement translates into  
the problem of constructing $Z_l$, such that 
\begin{equation}\label{eqn:approx} 
\Span  (Z_l) = \Range ( M ),  
\end{equation}  
where $M \approx \left( I - U_k U_k^T \right) D \in \mathbb{R}^{m \times p}$  
is some matrix with $\mbox{rank}(M) = l$. 
 
This is related to the standard task of constructing a low-rank  
approximation $M$ of $(I - U_k U_k^T) D$. It is well known, e.g.,~\cite{GVL-book},  
that an \textit{optimal} rank-$l$ approximation of $(I - U_k U_k^T) D$ is given by $M = X_l S_l Y_l^T$,  
where $S_l$ is a diagonal matrix of $l$ dominant singular values of $(I - U_k U_k^T) D$, and 
$X_l \in \mathbb{R}^{m \times l}$ and $Y_l \in \mathbb{R}^{p \times l}$ are the matrices of the corresponding  
left and right singular vectors, i.e.,  
\begin{equation}\label{eqn:svd_doc} 
(I - U_k U_k^T) D   Y_l   =   X_l S_l, \qquad D^T (I - U_k U_k^T)^T  X_l   =   Y_l S_l \;. 
\end{equation} 
Thus, it is natural to choose $Z_l = X_l$, i.e., augment $\bar U$ in~(\ref{eqn:new_U_doc})  
with a few left dominant singular vectors of $(I - U_k U_k^T) D$. It is clear that this choice  
of $Z_l$ indeed  
satisfies~(\ref{eqn:compl_U_doc}).  
Then, by Proposition~\ref{prop:H_Z} and~(\ref{eqn:svd_doc}), a run  
of SV-RR($A_D$, $\bar U$, $V$) 
with $\bar{U} = \left[ U_k, \ X_l \right]$ and $V$ in~(\ref{eqn:MUV_doc}) 
produces the projected matrix 
\begin{equation}\label{eqn:H_sv} 
H = \left[ 
\begin{array}{cc} 
\Sigma_k    &    U_k^T D \\ 
0           &    S_l Y^T_l        
\end{array} 
\right] \in \mathbb{R}^{(k+l) \times (k+p)} \; . 
\end{equation} 
This leads to the following updating scheme based on the singular vectors (SV) of $(I - U_k U_k^T) D$.  
 
\begin{algorithm}[Adding documents (SV)]\label{alg:doc_sv} 
\emph{Input:} $\Sigma_k$, $U_k$, $V_k$, $D$. \emph{Output:} $\tilde \Sigma_k$, $\tilde U_k$, $\tilde V_k$. 
\begin{enumerate}  
\item \emph{Compute} $l$ \emph{largest singular triplets} of $(I - U_k U_k^T) D$ 
\emph{in}~(\ref{eqn:svd_doc}).  
\item \emph{Construct} $H$ \emph{in}~(\ref{eqn:H_sv}).  
\emph{Compute the matrices}  
$\Theta_k$, $F_k$, \emph{and} $G_k$ \emph{that correspond to the} $k$ \emph{dominant singular triplets of} $H$. 
\item \emph{Set} 
$ 
\tilde \Sigma_k = \Theta_k, \ \tilde U_k = [U_k, \ X_l] F_k, \ \emph{and} \ \tilde V_k =  
\left[ 
\begin{array}{cc} 
 V_k & 0 \\  
 0   & I_p  
\end{array} 
\right] G_k$. 
\end{enumerate} 
\end{algorithm} 
\vspace{0.1in} 
 
Observe that the projected matrix produced by Algorithm~\ref{alg:doc_sv} has 
significantly fewer rows than the one in Algorithm~{\ref{alg:doc_zha}},  
bringing the $\mathcal{O}((k+p)^3)$ complexity of step 2 to $\mathcal{O}((k+p)(k+l)^2)$, 
which is linear in $p$. From a practical point of view, note also 
that, in contrast to Algorithm~\ref{alg:doc_zha},  
$(I - U_k U_k^T) D$ in step 1 should not be explicitly constructed,  
because the factorization~(\ref{eqn:svd_doc}) can be obtained by an iterative  
procedure, e.g.,~\cite{SVDPACK, Hernandez.Roman.Tomas:08}, 
that accesses $(I - U_k U_k^T)D$ (and its transpose) through matrix-vector multiplications with $D$ (and $D^T$)  
and orthogonalizations against the columns of $U_k$. The triplets $S_l$, $X_l$, and $Y_l$ 
need to  be approximated only with modest accuracy.  
 
Similar considerations can be exploited for the case of adding new terms with the 
aim of constructing $Z_l$ in~(\ref{eqn:new_V_term})--(\ref{eqn:compl_V_term}) 
satisfying~(\ref{eqn:approx}) with $M \approx \left( I - V_k V_k^T \right) T^T \in \mathbb{R}^{n \times p}$  
and $\mbox{rank}(M) = l$. 
Let $S_l$, $X_l \in \mathbb{R}^{n \times l}$, and $Y_l \in \mathbb{R}^{p \times l}$  
now be the factors  of the rank-$l$ SVD approximation 
of the matrix $(I - V_k V_k^T) T^T$, i.e., 
%
\begin{equation}\label{eqn:svd_term} 
(I - V_k V_k^T) T^T  Y_l   =   X_l S_l, \quad T (I - V_k V_k^T)  X_l   =   Y_l S_l \; . 
\end{equation} 
An optimal rank-$l$ approximation of $(I - V_k V_k^T) T^T$ is then given by $M = X_l S_l Y_l^T$, 
and analogy with the previous case suggests that we choose $Z_l = X_l$. 
Thus, by Proposition~\ref{prop:H_Z_terms} and~(\ref{eqn:svd_term}), 
a run of SV-RR($A_T$, $U$, $\bar V$) with $U$ defined in~(\ref{eqn:MUV_term}) and $\bar V = [V_k, \ X_l]$  
gives the projected matrix 
\begin{equation}\label{eqn:H_sv_terms} 
H = \left[ 
\begin{array}{cc} 
\Sigma_k    &    0 \\ 
T V_k           &    Y_l S_l        
\end{array} 
\right] \in \mathbb{R}^{(k+p) \times (k+l)} \; . 
\end{equation} 
As a result, we obtain the following updating scheme. 
\begin{algorithm}[Adding terms (SV)]\label{alg:term_sv} 
\emph{Input:} $\Sigma_k$, $U_k$, $V_k$, $T$. \emph{Output:} $\tilde \Sigma_k$, $\tilde U_k$, $\tilde V_k$. 
\begin{enumerate}  
\item \emph{Compute} $l$ \emph{largest singular triplets} of $(I - V_k V_k^T) T^T$  
\emph{in}~(\ref{eqn:svd_term}). 
\item \emph{Construct} $H$ \emph{in}~(\ref{eqn:H_sv_terms}). 
\emph{Compute the matrices} $\Theta_k$, $F_k$, \emph{and} $G_k$ \emph{that correspond to the} $k$  
\emph{dominant singular triplets of} $H$. 
\item \emph{Set} 
$ 
\tilde \Sigma_k = \Theta_k, \ \tilde U_k =  
\left[ 
\begin{array}{cc} 
 U_k & 0 \\  
 0   & I_p  
\end{array} 
\right] F_k$, \emph{and} $\tilde V_k = [V_k, \ X_l] G_k$ . 
\end{enumerate} 
\end{algorithm} 
\vspace{0.1in}

Finally, in the case of correcting the term weights, we would like to specify  
$Z^{(1)}_{l_1}$ and $Z^{(2)}_{l_2}$ in~(\ref{eqn:new_U_cw})--(\ref{eqn:compl_UV_cw}),  
such that 
\begin{equation}\label{eqn:approx1} 
\Span  (Z^{(1)}_{l_1}) = \Range ( M_1 ), \quad \Span  (Z^{(2)}_{l_2}) = \Range ( M_2 ) \; ,  
\end{equation}  
where $M_1 \approx \left( I - U_k U_k^T \right) C \in \mathbb{R}^{m \times p}$,  
$M_2 \approx \left( I - V_k V_k^T \right) W^T \in \mathbb{R}^{n \times p}$; 
$\mbox{rank}(M_1) = l_1$ and $\mbox{rank}(M_2) = l_2$.  
By analogy with the previous cases, we let $S^{(1)}_{l_1}$, $X^{(1)}_{l_1} \in \mathbb{R}^{m \times l_1}$,  
$Y^{(1)}_{l_1} \in \mathbb{R}^{p \times l_1}$ be the matrices of the $l_1$ largest singular values and  
associated left and right singular vectors of $(I - U_k U_k^T) C$, 
\begin{equation}\label{eqn:svd_cw1} 
(I - U_k U_k^T) C  Y^{(1)}_{l_1}   =   X^{(1)}_{l_1} S^{(1)}_{l_1}, \quad 
C^T (I - U_k U_k^T) X^{(1)}_{l_1}   =   Y^{(1)}_{l_1} S^{(1)}_{l_1} \; ,  
\end{equation} 
and $S^{(2)}_{l_2}$, $X^{(2)}_{l_2} \in \mathbb{R}^{n \times l_2}$, $Y^{(2)}_{l_2} \in \mathbb{R}^{p \times l_2}$  
correspond to the $l_2$ dominant triplets of $(I - V_k V_k^T) W^T$, 
\begin{equation}\label{eqn:svd_cw2} 
(I - V_k V_k^T) W^T  Y^{(2)}_{l_2}   =   X^{(2)}_{l_2} S^{(2)}_{l_2}, \quad 
W (I - V_k V_k^T) X^{(2)}_{l_2}  =   Y^{(2)}_{l_2} S^{(2)}_{l_2} \;. 
\end{equation} 
Optimal low-rank approximations of $(I - U_k U_k^T) C$ and $(I - V_k V_k^T) W^T$ 
are then given by $M_1 = X^{(1)}_{l_1} S_{l_1}^{(1)} Y_{l_1}^{(1)T}$ 
and $M_2 = X^{(2)}_{l_2} S_{l_2}^{(2)} Y_{l_2}^{(2)T}$, respectively. 
Hence, we choose $Z^{(1)}_{l_1} = X^{(1)}_{l_1}$ and $Z^{(2)}_{l_2} = X^{(2)}_{l_2}$. 
In this case, by Proposition~\ref{prop:H_Z_terms_cw} and~(\ref{eqn:svd_cw1})--(\ref{eqn:svd_cw2}),  
a run of SV-RR($A_{CW}$, $\bar U$, $\bar V$) with $\bar U = [U_k, \ X^{(1)}_{l_1}]$ and 
$\bar V = [V_k, \ X^{(2)}_{l_2}]$ 
yields the projected matrix 
\begin{equation}\label{eqn:H_sv_cw} 
H =  
\left[ 
\begin{array}{cc} 
 \Sigma_k & 0 \\  
 0        & 0  
\end{array} 
\right] + 
\left[ 
\begin{array}{c} 
 U_k^T C  \\  
 S^{(1)}_{l_1} Y^{(1)T}_{l_1}    
\end{array} 
\right]  
\left[  
W V_k, \ Y^{(2)}_{l_2} S^{(2)}_{l_2}  
\right] \in \mathbb{R}^{(k+l_1) \times (k+l_2)} \;, 
\end{equation} 
and can be summarized as the following updating algorithm. 
%
\begin{algorithm}[Correcting weights (SV)]\label{alg:cw_sv} 
\emph{Input:} $\Sigma_k$, $U_k$, $V_k$, $C$, $W$.  
\emph{Output:} $\tilde \Sigma_k$, $\tilde U_k$, $\tilde V_k$. 
\begin{enumerate}  
\item \emph{Compute} $l_1$ \emph{largest singular triplets of} $(I - U_k U_k^T) C$  
\emph{in}~(\ref{eqn:svd_cw1}) \emph{and} $l_2$  
\emph{largest singular triplets of} $(I - V_k V_k^T) W^T$~\emph{in}~(\ref{eqn:svd_cw2}) . 
\item \emph{Construct} $H$ \emph{in}~(\ref{eqn:H_sv_cw}). 
\emph{Compute the matrices} $\Theta_k$, $F_k$, \emph{and} $G_k$ \emph{that correspond to the} $k$ \emph{dominant singular  
triplets of} $H$. 
\item \emph{Set} 
$\tilde \Sigma_k = \Theta_k$,  
$\tilde U_k = [U_k, \ X^{(1)}_{l_1}] F_k$, \emph{and} $\tilde V_k = [V_k, \ X^{(2)}_{l_2}] G_k$ . 
\end{enumerate} 
\end{algorithm}

\subsection{The Golub--Kahan--Lanczos vectors}\label{subsec:gkl} 
 
In this subsection, we consider an alternative option for choosing the orthonormal 
vectors $Z_l$, $Z^{(1)}_{l_1}$, and $Z^{(2)}_{l_2}$.  
Instead of using the dominant singular triplets of the four respective matrices $(I - U_k U_k^T) D$,  
$(I - V_k V_k^T) T^T$, $(I - U_k U_k^T) C$, and $(I - V_k V_k^T) W^T$, under consideration,  
we propose to exploit for the same purpose 
a few basis vectors computed from  the Golub--Kahan--Lanczos 
(GKL) bidiagonalization procedure~\cite{Golub.Kahan:65, GVL-book}.   
 
Given an arbitrary $m$-by-$n$ matrix $A$, the GKL procedure constructs the orthonormal bases  
$P_l = [p_1, \ldots , p_l ]$ and $Q_{l+1} = [q_1, \ldots, q_{l+1}]$  
of the \textit{Krylov subspaces}  
$\mathcal{K}_{l}(AA^T,Aq_1) = $  
$\Span \left\{ A q_1, (A A^T)Aq_1, \ldots (A A^T)^{l-1} A q_1 \right\} \; 
$  
and 
$\mathcal{K}_{l+1}(A^T A, q_1) =$ \\  
$\Span \left\{ q_1, (A^T A)q_1, \ldots (A^T A)^{l} q_1 \right\} \;$, 
respectively; and an (upper) bidiagonal matrix $\underbar B_{l} \in \IR^{l \times (l+1)}$.   
The matrices $P_l$, $Q_{l+1}$, and $\underbar B_{l}$ are related by the fundamental  
identity 
\begin{equation}\label{eqn:GKL} 
\left\{ 
\begin{array}{ccc} 
A Q_l & = & P_l B_l \; , \\ 
A^{T} P_l & = & Q_{l+1} \underbar{B}^{T}_l \; ; 
\end{array} 
\right. 
\end{equation} 
where $B_l \in \mathbb{R}^{l \times l}$ is obtained from $\underbar {B}_l$ by removing the last column. 
The vectors $p_i$ and $q_i$ that comprise the matrices $P_l$ and $Q_l$ are called the \textit{left 
and right GKL vectors}, respectively.  
An implementation of the procedure is given by the algorithm below,  
which we further refer to as GKL($A$, $l$). 
\begin{algorithm}[GKL($A$, $l$)]\label{alg:GKL} 
\emph{Input:} $A \in \mathbb{R}^{m \times n}$, $l$.  
\emph{Output:} $\underbar{B}_l$, $P_l$, $Q_{l+1}$. 
\begin{enumerate}  
\item \emph{Choose} $q_1$, $\| q_1 \| = 1$. \emph{Set} $\beta_0 = 0$ . 
\item  \emph{\textbf{For}} $i = 1, \ldots, l$ \emph{\textbf{do}}  
  \item \hspace{0.2in}  
         $p_i = Aq_i - \beta_{i-1} p_{i-1}$ 
  \item  \hspace{0.2in}  
      \emph{\textbf{If}} $(m < n)$ \emph{Orthogonalize} $p_{i}$ \emph{against} $[p_1, \ldots, p_{i-1}]$. 
  \item  \hspace{0.2in}  
        $\alpha_i = \|p_i\|$. 
  \item  \hspace{0.2in}  
        $p_i = p_i / \alpha_i$. 
  \item  \hspace{0.2in}  
        $q_{i+1} = A^T p_i - \alpha_i q_i$. 
  \item  \hspace{0.2in}  
      \emph{\textbf{If}} $(m \geq n)$ \emph{Orthogonalize} $q_{i+1}$ \emph{against} $[q_1, \ldots, q_{i}]$. 
  \item  \hspace{0.2in}  
        $\beta_i = \|q_{i+1}\|$. 
  \item  \hspace{0.2in}  
        $q_{i+1} = q_{i+1} / \beta_i$. 
\item \emph{\textbf{EndFor}} 
\item \emph{Return}  
$\underbar{B}_l = \text{\emph{diag}}\left\{\alpha_1, \ldots, \alpha_l\right\} + \text{\emph{superdiag}}\left\{\beta_1, \ldots, \beta_l \right\}$,  
$P_l = \left[ p_1, \ldots, p_l \right] $, \emph{and}  
$Q_{l+1} = [q_1, \ldots, q_{l+1}]$.  
\end{enumerate} 
\end{algorithm} 
\vspace{0.1in} 
%


Note that, in exact arithmetic, steps 4 and 8 of the algorithm  
are unnecessary,  
i.e., orthogonality of vectors $p_i$ and $q_i$ is ensured solely  
by the short-term recurrences in steps 3 and 6.  
In the presence of the round-off, it is generally required to 
reorthogonalize both $p_i$ and $q_i$ against all previously constructed vectors.  
However, in Algorithm~\ref{alg:GKL} we follow the observation  
in~\cite{Simon.Zha:00} 
suggesting that it is possible to reorthogonalize only one of the  
vector sets without significant loss of orthogonality for the other. 
 
Let us first address the case of adding new documents, and assume that  
$P_l \in \mathbb{R}^{m \times l}$, $Q_{l+1} \in \mathbb{R}^{p \times (l+1)}$,  
and $\underbar{B}_l$ are produced by a run of GKL($(I - U_k U_k^T) D$, $l$).  
Then as a rank-$l$ approximation of $(I - U_k U_k^T) D$ we choose 
$M = P_l \underbar{B}_l Q_{l+1}^T$ and set $Z_l = P_l$, 
i.e., we define $\bar U$  in (\ref{eqn:new_U_doc}) by 
augmenting $U_k$ with several left GKL vectors of $(I - U_k U_k^T) D$.  
 
In contrast to the case of the singular vectors, this choice does not provide  
an optimal rank-$l$ approximation. The optimality, however, is traded 
for a simpler and faster procedure to construct $Z_l$.  
Evaluation of the performance of the 
GKL vectors as opposed to the singular vectors in the context 
of the dimension reduction  
can be found in~\cite{ChenSaad:FiltLan}.

By~(\ref{eqn:GKL}), we have  
\begin{equation}\label{eqn:gkl_doc} 
\left\{ 
\begin{array}{lll} 
(I - U_k U_k^T) D Q_l &  =  & P_l B_l \\ 
D^T (I - U_k U_k^T) P_l &  =  & Q_{l+1} \underbar{B}_l^T \;. 
\end{array} 
\right. 
\end{equation} 
It can be seen from the above relation that $Z_l = P_l$ satisfies~(\ref{eqn:compl_U_doc}). 
Thus, by Proposition~\ref{prop:H_Z}, the projected matrix produced by an application of  
SVD-RR($A_D$, $\bar U$, $V$), where $\bar U = [U_k, \ P_l]$ and $V$ is defined in~(\ref{eqn:MUV_doc}),  
takes the form  
\begin{equation}\label{eqn:H_gkl} 
H = \left[ 
\begin{array}{cc} 
\Sigma_k    &    U_k^T D \\ 
0           &    \underbar{B}_l Q^T_{l+1}        
\end{array} 
\right] \in \mathbb{R}^{(k+l) \times (k+p)} \; . 
\end{equation} 
%
%
As a result,  
we obtain the following updating algorithm. 
 
\begin{algorithm}[Adding documents (GKL)]\label{alg:doc_gkl} 
\emph{Input:} $\Sigma_k$, $U_k$, $V_k$, $D$. \emph{Output:} $\tilde \Sigma_k$, $\tilde U_k$, $\tilde V_k$. 
\begin{enumerate}  
\item \emph{Apply GKL($(I - U_k U_k^T) D$, $l$)}, \emph{given by Algorithm~\ref{alg:GKL}},  
\emph{to produce} $P_l$, $Q_{l+1}$, \emph{and} $\underbar{B}_l$ \emph{satisfying}~(\ref{eqn:gkl_doc}). 
\item \emph{Construct} $H$ \emph{in}~(\ref{eqn:H_gkl}). 
\emph{Compute the matrices} $\Theta_k$, $F_k$, \emph{and} $G_k$ \emph{that correspond to the} $k$  
\emph{dominant singular triplets of} $H$. 
\item \emph{Set} 
$ 
\tilde \Sigma_k = \Theta_k, \ \tilde U_k = [U_k, \ P_l] F_k, \ \emph{and} \ \tilde V_k =  
\left[ 
\begin{array}{cc} 
 V_k & 0 \\  
 0   & I_p  
\end{array} 
\right] G_k$. 
\end{enumerate} 
\end{algorithm} 
\vspace{0.1in} 
 
Note that  applying the GKL procedure in step 1 of Algorithm~\ref{alg:doc_gkl} can be carried out  
without explicitly constructing the matrix $(I - U_k U_k^{T})D$ and its transpose. Instead, 
the matrices can be accessed through matrix-vector multiplications with $D$ or $D^T$, and  
orthogonalizations against the columns of $U_k$. 
 
If  new terms $T$ are added to the document collection, then $l$ steps of the GKL procedure 
should be applied to the matrix $(I - V_k V_k^T) T^T$. This leads to the identity  
\begin{equation}\label{eqn:gkl_term} 
\left\{ 
\begin{array}{lll} 
(I - V_k V_k^T) T^T Q_l &  =  & P_l B_l \\ 
T (I - V_k V_k^T) P_l &  =  & Q_{l+1} \underbar{B}_l^T \;, 
\end{array} 
\right. 
\end{equation} 
where $P_l \in \mathbb{R}^{n \times l}$ and $Q_{l+1} \in \mathbb{R}^{p \times (l+1)}$  
are the matrices of the left and right GKL vectors 
of $(I - V_k V_k^T)T^T$. Similar to the previous case, we approximate 
$(I - V_k V_k^T)T^T$ by $M = P_l \underbar{B}_l Q_{l+1}^T$, and set $Z_l = P_l$. 
Then, combining Proposition~\ref{prop:H_Z_terms} and~(\ref{eqn:gkl_term}), 
we obtain the projected matrix 
\begin{equation}\label{eqn:H_gkl_term} 
H = \left[ 
\begin{array}{cc} 
\Sigma_k    &    0 \\ 
T V_k       &    Q_{l+1} \underbar{B}_l^T         
\end{array} 
\right] \in \mathbb{R}^{(k+p) \times (k+l)} \; , 
\end{equation} 
which is produced by SVD-RR($A_T$, $U$, $\bar V$) with 
$U$ in~(\ref{eqn:MUV_term}) and $\bar V = [V_k, \ P_l]$.  
The resulting updating scheme is summarized in the algorithm below. 
 
\begin{algorithm}[Adding terms (GKL)]\label{alg:term_gkl} 
\emph{Input:} $\Sigma_k$, $U_k$, $V_k$, $D$. \emph{Output:} $\tilde \Sigma_k$, $\tilde U_k$, $\tilde V_k$. 
\begin{enumerate}  
\item \emph{Apply GKL($(I - V_k V_k^T) T^T$, $l$)}, \emph{given by Algorithm~\ref{alg:GKL}},  
\emph{to produce} $P_l$, $Q_{l+1}$, \emph{and} $\underbar{B}_l$ \emph{satisfying}~(\ref{eqn:gkl_term}). 
\item \emph{Construct} $H$ \emph{in}~(\ref{eqn:H_gkl_term}). 
\emph{Compute the matrices} $\Theta_k$, $F_k$, \emph{and} $G_k$ \emph{that correspond to the} $k$  
\emph{dominant singular triplets of} $H$. 
\item \emph{Set} 
$ 
\tilde \Sigma_k = \Theta_k, \ \tilde U_k =  
\left[ 
\begin{array}{cc} 
 U_k & 0 \\  
 0   & I_p  
\end{array} 
\right] F_k$, \emph{and} $\tilde V_k = [V_k, \ P_l] G_k$ . 
\end{enumerate} 
\end{algorithm} 
\vspace{0.1in}

Finally, if the term weights are corrected, the GKL procedure is applied to both  
$(I - U_k U_k^T) C$ and $(I - V_k V_k^T) W$. A run of GKL($(I - U_k U_k^T) C$, $l_1$) 
gives the equality   
\begin{equation}\label{eqn:gkl_cw1} 
\left\{ 
\begin{array}{lll} 
(I - U_k U_k^T) C Q_{l_1}^{(1)} &  =  & P_{l_1}^{(1)} B_{l_1}^{(1)} \\ 
C^T (I - U_k U_k^T) P_{l_1}^{(1)} &  =  & Q_{l_1+1}^{(1)} \underbar{B}_{l_1}^{(1)T} \;, 
\end{array} 
\right. 
\end{equation} 
and GKL($(I - V_k V_k^T) W$, $l_2$) results in  
\begin{equation}\label{eqn:gkl_cw2} 
\left\{ 
\begin{array}{lll} 
(I - V_k V_k^T) W^T Q_{l_2}^{(2)} &  =  & P_{l_2}^{(2)} B_{l_2}^{(2)} \\ 
W (I - V_k V_k^T) P_{l_2}^{(2)} &  =  & Q_{l_2+1}^{(2)} \underbar{B}_{l_2}^{(2)T} \;. 
\end{array} 
\right. 
\end{equation} 
In~(\ref{eqn:gkl_cw1})--(\ref{eqn:gkl_cw2}), the columns of the matrix pairs  
$P^{(1)}_{l_1} \in \mathbb{R}^{m \times l_1}$, $Q^{(1)}_{l_1 + 1} \in \mathbb{R}^{p \times (l_1 + 1)}$  
and $P^{(2)}_{l_2} \in \mathbb{R}^{n \times l_2}$, $Q^{(2)}_{l_2 + 1} \in \mathbb{R}^{p \times (l_2 + 1)}$  
correspond to the left and right GKL vectors of $(I - U_k U_k^T) C$ and $(I - V_k V_k^T) W$, 
respectively. We then can approximate $(I - U_k U_k^T) C$ by  
$M_1 = P^{(1)}_{l_1} \underbar{B}^{(1)}_{l_1} Q_{l_1+1}^{(1)T}$ and $(I - V_k V_k^T) W$   
by $M_2 = P^{(2)}_{l_2} \underbar{B}^{(2)}_{l_2} Q_{l_2+1}^{(2)T}$, choosing  
$Z^{(1)}_{l_1} = P^{(1)}_{l_1}$ and $Z^{(2}_{l_2} = P^{(2)}_{l_2}$. Hence, 
by Proposition~\ref{prop:H_Z_terms_cw} and~(\ref{eqn:gkl_cw1})--(\ref{eqn:gkl_cw2}), 
a run of SVD-RR($A_{CW}$, $\bar{U}$, $\bar{V}$) with $\bar U = [U_k, \ P^{(1)}_{l_1}]$ 
and $\bar V = [V_k, \ P^{(2)}_{l_2}]$ yields the projected matrix  
\begin{equation}\label{eqn:H_gkl_cw} 
H =  
\left[ 
\begin{array}{cc} 
 \Sigma_k & 0 \\  
 0        & 0  
\end{array} 
\right] + 
\left[ 
\begin{array}{c} 
 U_k^T C  \\  
 \underbar{B}_{l_1}^{(1)} Q_{l_1+1}^{(1)T}   
\end{array} 
\right]  
\left[  
W V_k, \ Q_{l_2+1}^{(2)} \underbar{B}_{l_2}^{(2)T} 
\right] \in \mathbb{R}^{(k+l_1) \times (k+l_2)} \;. 
\end{equation} 
This suggests the following updating scheme. 
\begin{algorithm}[Correcting weights (GKL)]\label{alg:cw_gkl} 
\emph{Input:} $\Sigma_k$, $U_k$, $V_k$, $D$. \emph{Output:} $\tilde \Sigma_k$, $\tilde U_k$, $\tilde V_k$. 
\begin{enumerate}  
\item \emph{Apply GKL($(I - U_k U_k^T) C$, $l_1$) and GKL($(I - V_k V_k^T) W^T$, $l_2$)},  
\emph{given by Algorithm~\ref{alg:GKL}},  
\emph{to produce the GKL vectors} \emph{satisfying}~(\ref{eqn:gkl_cw1}) \emph{and} (\ref{eqn:gkl_cw2}). 
\item \emph{Construct} $H$ \emph{in}~(\ref{eqn:H_gkl_cw}). 
\emph{Compute the matrices} $\Theta_k$, $F_k$, \emph{and} $G_k$ \emph{that correspond to the} $k$  
\emph{dominant singular triplets of} $H$. 
\item \emph{Set} 
$ 
\tilde \Sigma_k = \Theta_k, \ \tilde U_k = [U_k, \ P_{l_1}^{(1)}] F_k, \ \emph{and} \ 
V_k = [V_k, \ P_{l_2}^{(2)}] G_k.  
$ 
\end{enumerate} 
\end{algorithm} 
\vspace{0.1in} 
 

\subsection{Complexity analysis}\label{subsec:complexity}

We start by addressing the cost of Algorithm~\ref{alg:doc_sv} that performs the  
updating  after adding new documents. 
To be specific, 
we assume that a multiple of $l$ iterations of a GKL-based singular value solver,  
e.g.,~\cite{Hernandez.Roman.Tomas:08}, are used to determine the $l$ dominant singular triplets of $(I - U_k U_k^T) D$.  
In this case, the cost of step 1 of the algorithm can be estimated as  
$\mathcal{O} (l  \NNZ (D) + lmk + l^2 (m + p))$, where $\NNZ(D)$ is the number of non-zero elements in $D$.  
Here the first two terms come from the matrix-vector multiplications with $(I - U_k U_k^T) D$.  
In particular, the first term is given by the Sparse Matrix-Vector multiplications (SpMVs) with $D$  
and $D^T$, and the second term is contributed by orthogonalizations against the columns of $U_k$. 
The last term accounts for reorthogonalizations of the GKL vectors and the final SV-RR procedure to  
extract the singular triplet approximations from the corresponding Krylov subspaces.  
 
The complexity of step 2 of Algorithm~\ref{alg:doc_sv} is $\mathcal{O}((k+p)(k+l)^2 + \NNZ(D) k + lp)$. 
The first term represents the cost of the SVD of the matrix $H$ in~(\ref{eqn:H_sv}). 
The remaining terms are given by the construction of $U_k^T D$  
(utilizing SpMVs with $D$) and the multiplication of $Y_l^T$ by the diagonal matrix $S_l$  
to form the $(1,2)$ and $(2,2)$ blocks of $H$, respectively.  
The cost of step 3 is $\mathcal{O}(k^2 (m + n) + lmk)$. 
 
The complexity of Algorithm~\ref{alg:doc_gkl} that also addresses the case of adding new 
documents and extends Algorithm~\ref{alg:doc_sv} by  
replacing the singular vectors of $(I - U_k U_k^T) D$ with the GKL vectors is similarly estimated. 
The difference in cost of the two schemes essentially comes from their initial step.  
In particular, the complexity of step 1 of Algorithm~\ref{alg:doc_gkl} is $\mathcal{O}(l  \NNZ (D) + lmk + l^2 p)$,  
i.e., it requires $\mathcal{O}(l^2 m)$ less operations than the corresponding step of Algorithm~\ref{alg:doc_sv}.  
This cost reduction is due to avoiding the SV-RR procedure that is invoked by a GKL-based singular value solver in 
step 1 of Algorithm~\ref{alg:doc_sv}.  
 
The overall complexities of Algorithms~\ref{alg:doc_sv} and~\ref{alg:doc_gkl}, 
as well as of Algorithm~\ref{alg:doc_zha} addressed at the beginning of Section~\ref{sec:new_updates},  
are summarized in Table~\ref{tbl:doc_complex}.  
%
 
{\small 
\begin{table} 
\begin{centering} 
\begin{tabular}{|c|c|} 
\hline  
Algorithm & Complexity\tabularnewline 
\hline  
\hline  
\multicolumn{1}{|l|}{Alg.~\ref{alg:doc_zha} (ZS)} & $(k+p)^{3}+mp{}^{2}+mpk+k^{2}(m+n)$  
\tabularnewline 
\hline  
\multicolumn{1}{|l|}{Alg.~\ref{alg:doc_sv} (SV)} & $(k+p)(k+l)^{2}+ml^{2}+mlk+k^{2}(m+n)+\NNZ(D)(l+ k)$\tabularnewline 
\hline  
\multicolumn{1}{|l|}{Alg.~\ref{alg:doc_gkl} (GKL)} & $(k+p)(k+l)^{2}+mlk+k^{2}(m+n)+\NNZ(D)(l + k)$ 
\tabularnewline 
\hline  
\end{tabular} 
\par\end{centering} 
\caption{Asymptotic complexity of different LSI updating schemes for adding new documents.} 
\label{tbl:doc_complex} 
\end{table} 
} 
The main observation drawn from Tables~\ref{tbl:doc_complex}--\ref{tbl:cw_complex} is that,  
unlike the Zha--Simon approaches, the new updating schemes no longer exhibit the cubic scaling in $p$.     
It can be seen that the proposed algorithms scale linearly in the update size, leading to 
significant computational savings, especially in the context of large text collections, 
as demonstrated in our TREC8 example in Section~\ref{sec:num}.  
 
We also note that the new schemes allow one to take advantage of the sparsity of the update,  
and require less storage due to working with sparse or thinner dense matrices. 
For example, in the case of appending new documents, in addition to $\Sigma_k$, 
$U_k$, $V_k$, and $D$, the Zha--Simon algorithm requires storing $mp$ elements
of the orthogonal basis from the QR decomposition~\eqref{eqn:qr_doc} and $(k+p)^2$ entries
of the projected matrix~$H_D$. In contrast, the proposed schemes only store $ml + pl$
elements of the $l$ singular (or GKL) vectors of $(I - U_k U_k^T) D$ and $(k+l)(k+p)$
entries of the reduced matrix $H$. Clearly, if $l << p$ then the reduction in storage
is significant.  
%


 
{\small 
\begin{table} 
\begin{centering} 
\begin{tabular}{|c|c|} 
\hline  
Algorithm & Complexity\tabularnewline 
\hline  
\hline  
\multicolumn{1}{|l|}{Alg.~\ref{alg:term_zha} (ZS)} & $(k+p)^{3}+np^{2}+npk+k^{2}(m+n)$  
\tabularnewline 
\hline  
\multicolumn{1}{|l|}{Alg.~\ref{alg:term_sv} (SV)} & $ (k+p)(k+l)^{2}+nl^{2}+nlk+k^{2}(m+n)+\NNZ(D)(l+ k) $ 
\tabularnewline 
\hline  
\multicolumn{1}{|l|}{Alg.~\ref{alg:term_gkl} (GKL)} & $(k+p)(k+l)^{2}+nlk+k^{2}(m+n)+ \NNZ(D)(l + k)$ 
\tabularnewline 
\hline  
\end{tabular} 
\par\end{centering} 
\caption{Asymptotic complexity of different LSI updating schemes for adding new terms.} 
\label{tbl:term_complex} 
\end{table} 
}

If $l = 0$ then the introduced algorithms turn into the early methods in~\cite{BeDuBri:95, OBrien:94}.  
As seen from Tables~\ref{tbl:doc_complex}--\ref{tbl:cw_complex} after setting $l = 0$, 
these methods can be extremely fast, but generally yield modest 
retrieval accuracy.  
This  is not surprising.  In the language of the present paper, 
the schemes~\cite{BeDuBri:95, OBrien:94} represent the SV-RR procedures  
where the left, right or  
both search subspaces stay unmodified during the matrix updates, e.g.,  
always $\bar{U} = U_k$ or \ 
$\bar{V} = V_k$, and are ``too small'' to provide satisfactory retrieval  
quality.   
 
If $l = p$ then the costs  of  
the proposed schemes resemble those of the  
Zha--Simon methods~\cite{Zha.Simon:99}.  
In this case, the search subspaces become sufficiently large to ensure the inclusion of the dominant  
singular subspaces and, similar to the algorithms in~\cite{Zha.Simon:99},  
produce  the exact triplets of $A_D$, $A_T$, or $A_{CW}$. 
As has been discussed, this fixes the problem of the deteriorating retrieval accuracy, however, it 
may result in loss of the efficiency if $p$ is large.

 
{\small 
\begin{table} 
\begin{centering} 
\begin{tabular}{|c|c|} 
\hline  
Algorithm & Complexity\tabularnewline 
\hline  
\hline  
\multicolumn{1}{|l|}{Alg.~\ref{alg:cw_zha} (ZS)} & $(k+p)^{3}+(m+n)(p^{2}+pk+k^{2})$  
\tabularnewline 
\hline  
\multicolumn{1}{|l|}{Alg.~\ref{alg:cw_sv} (SV)} & $(k+l)^{3}+(m+n)(l^{2}+lk+k^{2})+  
(\NNZ(C)+\NNZ(W))(l + k) $\tabularnewline 
\hline  
\multicolumn{1}{|l|}{Alg.~\ref{alg:cw_gkl} (GKL)} & $(k+l)^{3}+(m+n)(lk + k^{2})+ 
(\NNZ(C)+\NNZ(W))(l + k) $\tabularnewline 
\hline  
\end{tabular} 
\par\end{centering} 
\caption{Asymptotic complexity of different updating schemes for correcting the term weights. 
} 
\label{tbl:cw_complex} 
\end{table} 
} 
 
The updating methods of this paper can balance 
computational expenses by appropriate 
choices of $l$ and can be placed in between the two extremes that correspond to 
\cite{BeDuBri:95, OBrien:94} and~\cite{Zha.Simon:99}.  
In the next section, we demonstrate that $l$ can be chosen small  
without sacrificing retrieval accuracy.

\section{Numerical experiments}\label{sec:num} 
 
The goal of this section is to demonstrate the differences in cost and 
accuracy of the discussed updating  schemes, and to verify the results 
of the complexity analysis in subsection~\ref{subsec:complexity}. 
 
In  our  experiments  we  apply  the updating  algorithms  to  several 
standard document  collections, such  as MEDLINE, 
\iftoggle{LONG}
{
CRANFIELD,
}  
\ NPL, and 
TREC8. 
These datasets have sizes that  range from a few thousands of  terms  
and documents to 
hundreds  of  thousands,  and  are  commonly  used  to  benchmark  the 
performance of  text mining  techniques.  Each collection  is supplied 
with a  set of $n_q$ ``canonical'' queries  and expert-generated lists 
of  relevant documents  corresponding  to these  queries.  The  listed 
relevant documents  are the  ones that should  be ideally  returned in 
response to the query,  and therefore provide information necessary to 
evaluate the accuracy of automatic retrieval. 
 
The MEDLINE
\iftoggle{LONG}
{, CRANFIELD,}
\; and NPL term-document matrices are available from 
\url{ftp.cs.cornell.edu/pub/smart/}. 
The TMG software~\cite{TMG-tool} has been used to parse the TREC8 dataset and generate  
the term-document matrix. The standard pre-processing has included stemming and deleting  
common words according to the TMG-provided stop-list. Terms with no more than 5 occurrences  
or with appearances in more than $100,000$ documents have been removed and 125 empty documents  
have been ignored.  
In all tests the weighting schemes for documents and queries have been set to  
\texttt{lxn.bpx}~\cite{Kolda.OLeary:98, Salton:89}. 
 
Due to space limitations we present results only for the case 
of adding new documents. Our experience with the other update types has led to similar 
observations and conclusions.

The tests are organized as follows. Given a term-document matrix $A$, we fix its first $t$ columns,  
and compute the $k$ dominant singular triplets $\Sigma_k$, $U_k$, and $V_k$ of the corresponding  
submatrix of $A$. This submatrix, denoted using the {\sc matlab} array-slicing 
notation by $A(:,1:t)$, represents  
the initial state of the document collection. To simulate the arrival of new documents, the remaining  
$n - t$ columns are consecutively added in groups of $p$ columns. Thus, for the first update, we  
append the submatrix $A(:,t+1:t+p)$, for the next one $A(:, t+p+1 : t+2p)$, etc. 
As a result, in our experiments, we nearly  double (and sometimes triple) the  
initial size $t$ of the collection. 
To track the effects of the update size on the  
efficiency of the updating schemes, for each 
dataset we consider several different values of $p$.  
 
After adding each column group, we update the $k$ singular triplets using the proposed  
Algorithm~\ref{alg:doc_sv} (``SV'') and Algorithm~\ref{alg:doc_gkl} (``GKL''), as well as the 
existing scheme of Zha and Simon in Algorithm~\ref{alg:doc_zha} (``ZS''). Thereafter,  
the similarity scores~(\ref{eqn:scores})  
with $\alpha = 0$ are evaluated for each ``canonical'' query, and  
the average precisions are calculated using the standard $N$-point formula,  
e.g.,~\cite{Harman:95, Kolda.OLeary:98}.  
Since a total of $n_q$ ``canonical'' queries are provided for each  
dataset, we calculate  
the \textit{mean} of the corresponding $n_q$ average precisions and plot the resulting quantity against  
the current number of documents in the collection. We also measure time spent for each update.  
 
All experiments have been performed  in {\sc matlab}.  Note that, with 
this testing environment,  little or nothing can be  deduced about the 
methods'  capabilities  in terms  of  the \textit{absolute}  execution 
times.  However, the timing  results can provide illustrations for the 
complexity                         findings                         in 
Tables~\ref{tbl:doc_complex}--\ref{tbl:cw_complex}    and    give    a 
comparison of the  cost of the schemes \textit{relative to each  other}.  
These comparisons are meaningful when the same conditions are used 
in each case, e.g., the same {\sc matlab} function  
is used for computing the partial SVD 
and the same optimization techniques, if any, are  applied.  
 
Due to  the established unifying framework  of the 
SV-RR procedure, our codes are  organized to follow the same execution 
path  and  differ  only  in  ``localized'' subtasks,  such  as,  e.g., 
computing   a  few   dominant   singular  triplets   in   step  1   of 
Algorithm~\ref{alg:doc_sv}   or  the   GKL  vectors   in  step   1  of 
Algorithm~\ref{alg:doc_gkl} instead  of the QR decomposition  of $(I - 
U_k  U_k^T) D$  in the  original Algorithm~\ref{alg:doc_zha}.  All the 
linear    algebra   operations,    such    as   dense    matrix-matrix 
multiplications,  the SVD  and  QR decompositions,  SpMVs, etc.,  are 
accomplished by the same {\sc  matlab} routines in all of the updating 
schemes. 
 
In step 1 of Algorithm~\ref{alg:doc_gkl}, we use our own straightforward {\sc matlab} implementation of the GKL procedure,  
which is based on Algorithm~\ref{alg:GKL}. Our singular value solver  
in step 1 of Algorithm~\ref{alg:doc_sv} is built on top of this GKL procedure by  
additionally performing the SV-RR computation with respect to the left and right GKL vectors.  
The latter amounts to the SVD of a bidiagonal matrix followed by the construction of the Ritz  
singular vectors. The convergence of the singular triplets is achieved, and the solver is stopped,  
after the difference between the sums of the $k$ dominant singular value approximations on two consecutive  
iterations becomes smaller than $10^{-1}$. The maximum number of iterations is set to $p$. 
 
The starting vector in the GKL algorithm 
is always set to the vector of all ones and then normalized to have unit length. We have observed that  
this choice often leads to a higher retrieval accuracy compared to a random vector.   
%
The reduced dimensions  $k$ are set to (nearly) optimal, in the retrieval accuracy, 
values   as  observed  in~\cite[Figures   5-6]{ChenSaad:FiltLan}.  The 
parameter $l$ that  determines the number of singular triplets in 
step 1 of Algorithm~\ref{alg:doc_sv} or  the number of the GKL vectors 
in step  1 of Algorithm~\ref{alg:doc_gkl} is  experimentally chosen to 
provide  balance  between  the  computational  cost  and  retrieval 
accuracy. 
In particular, for all examples except for TREC8, $l$ is
chosen to be the smallest value that leads to the retrieval accuracy comparable
to that of the standard algorithm.  
 
\begin{figure}[h] 
 \begin{center} 	 
 \includegraphics[width = 5.5cm]{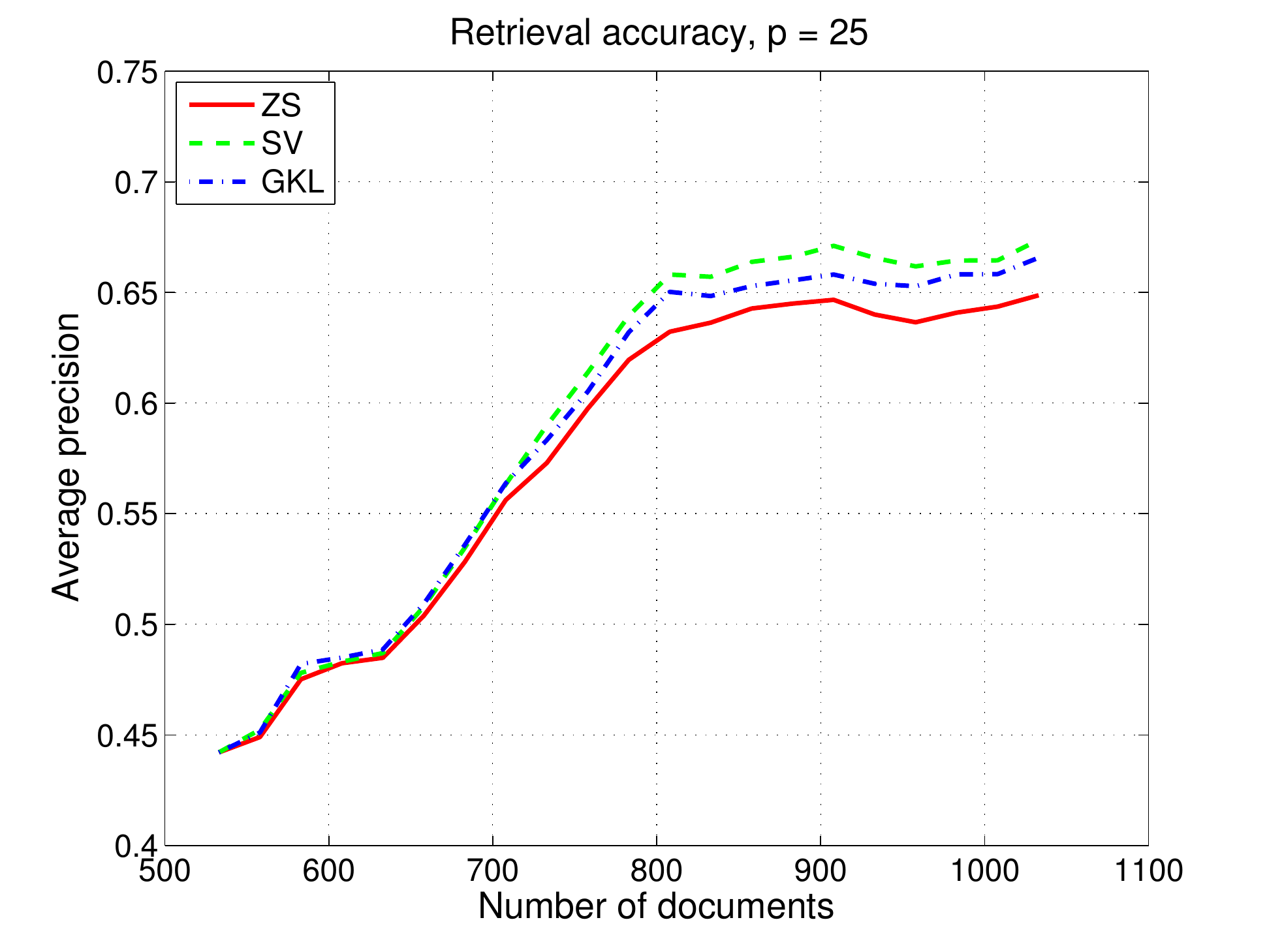} 
 \includegraphics[width = 5.5cm]{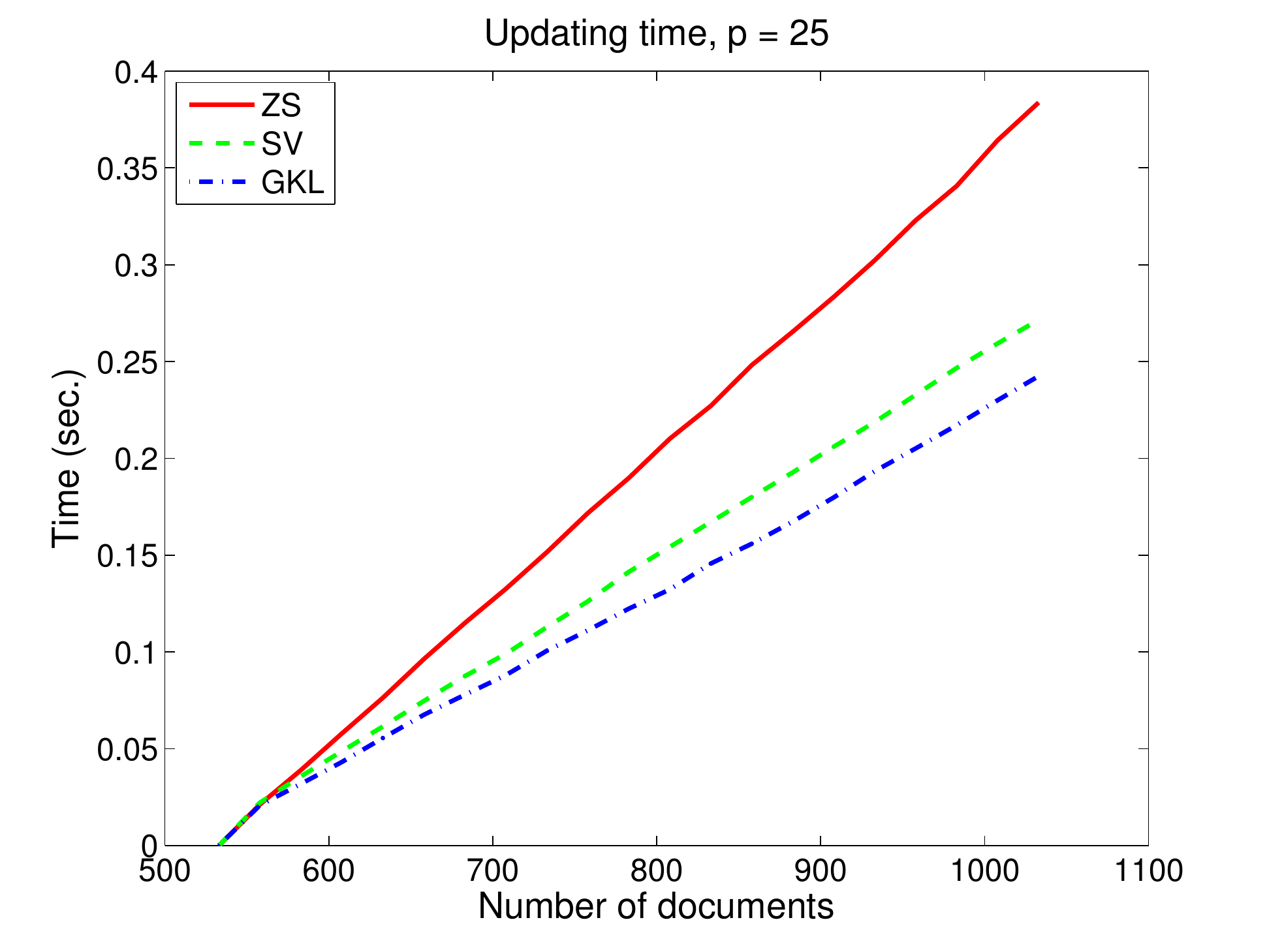} \\ 
 \includegraphics[width = 5.5cm]{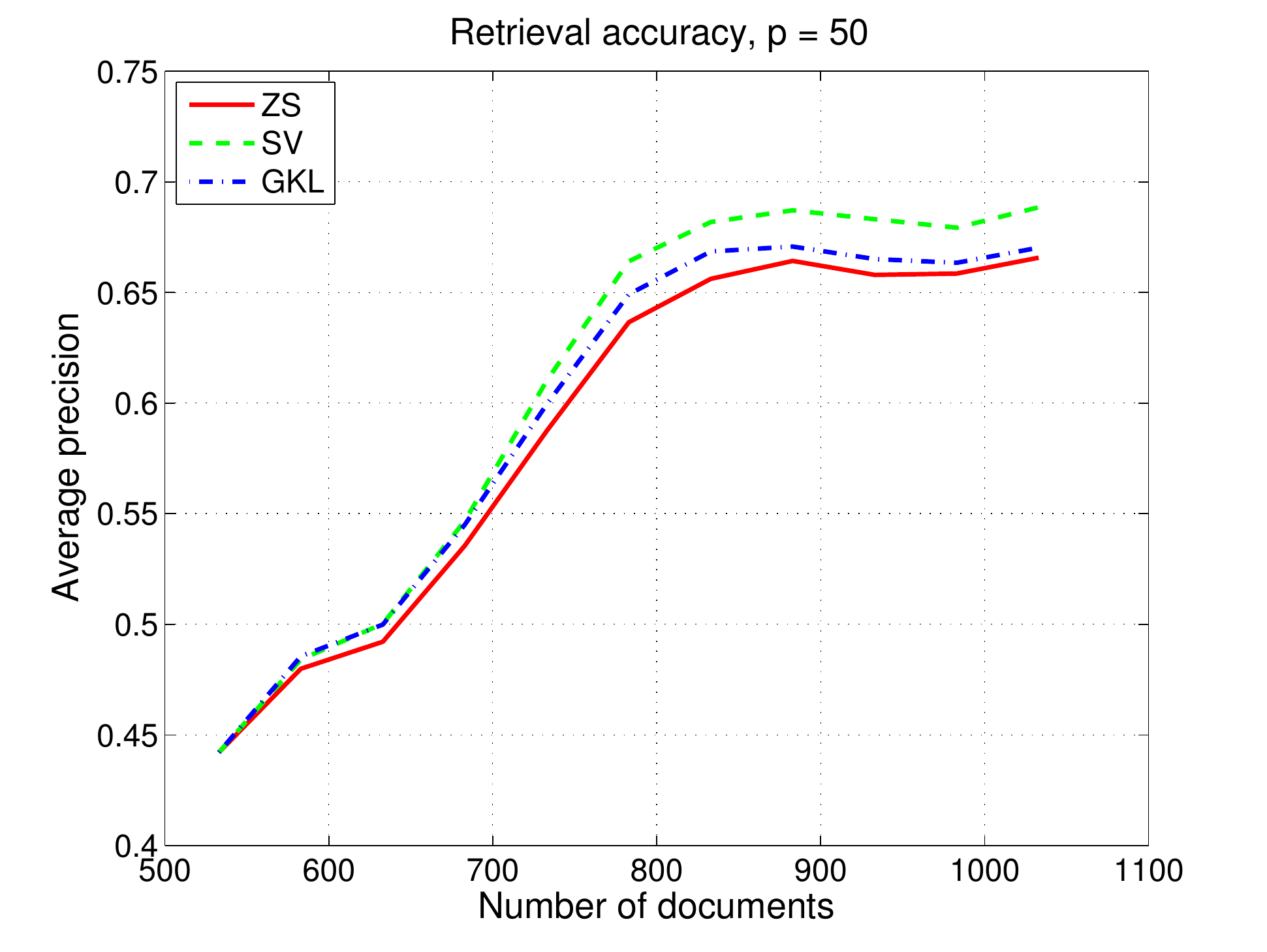} 
 \includegraphics[width = 5.5cm]{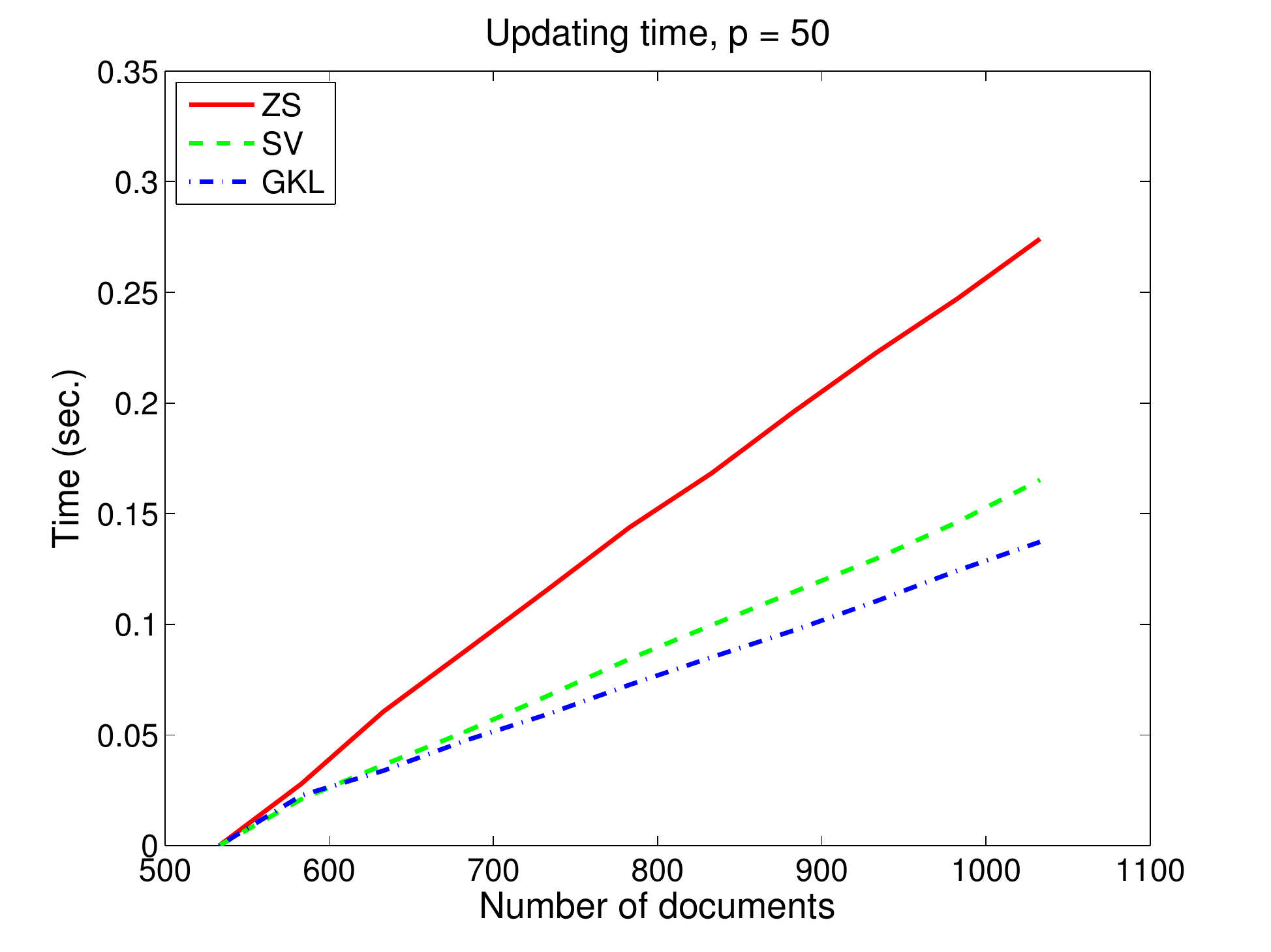}  
 \end{center} 
 \caption{MEDLINE collection: $m = 7,014$, $n = 1,033$, $k = 75$, $t = 533$, $n_q = 30$.  
The average precision and time for adding groups of $p = 25$ (top) and $p = 50$ (bottom)  
documents. The number of singular triplets of $(I - U_k U_k^T) D$ computed by Algorithm~\ref{alg:doc_sv} (``SV'') 
is $2$ (top) and $4$ (bottom). The number of GKL steps in Algorithm~\ref{alg:doc_gkl} (``GKL'')  
is $3$ (top) and $5$ (bottom).   
The methods are compared to Algorithm~\ref{alg:doc_zha} (``ZS'').
}\label{fig:med}  
\end{figure} 
 
Figure~\ref{fig:med} compares different updating schemes for the MEDLINE collection.  
This collection is known to be small, 
with the term-document matrix having $m = 7,014$ rows and $n = 1,033$ columns.  
We fix the initial $t = 533$ columns and add the rest in groups of $p = 25$ (top) and $p = 50$ (bottom).    
The size $k$ of the reduced subspace is set to $75$. 
 
The left-hand side plots  demonstrate the differences in the retrieval 
accuracy.   The plots  to the  right compare  the timing  results. The 
horizontal axes  represent the number  of documents in  the collection 
after consecutive  updates.  The vertical axes correspond  to the mean 
of  the average  precisions  (left) and  the \textit{cumulative}  time 
(right), i.e., the  total time spent by the  algorithms to perform the 
current and preceding updates. 
 
The example demonstrates that the number $l$ of the singular and  
GKL vectors generated by Algorithms~\ref{alg:doc_sv} and \ref{alg:doc_gkl} 
can be very small. In particular, we construct as few as 2-3 singular triplets and 3-4 GKL vectors.  
As        anticipated       from        the        discussion       in 
subsection~\ref{subsec:complexity},  this  fact   gives  rise  to   
updating  schemes that  are  significantly faster  than the  Zha--Simon 
approach in Algorithm~\ref{alg:doc_zha}. 
Figure~\ref{fig:med} (right) confirms that the new strategies in Algorithms~\ref{alg:doc_sv}  
and~\ref{alg:doc_gkl} are indeed considerably faster than the existing scheme.  
Remarkably, the gain in the efficiency comes without any loss of the retrieval accuracy; 
see Figure~\ref{fig:med} (left).  
 
It can be observed from the figure that Algorithm~\ref{alg:doc_sv} gives the  
most accurate results though it requires slightly more compute time than  
Algorithm~\ref{alg:doc_gkl}. This time difference is caused by the overhead  
accumulated by the singular value solver to ensure the convergence and perform  
the extraction of the approximate singular triplets in step 1 of Algorithm~\ref{alg:doc_sv}. 
%

{\small 
\begin{table} 
\begin{centering} 
\begin{tabular}{|l||c|c|c|c|c|c|}
\hline 
\# top ranked doc. & 10  & 30 & 40  & 70 & 500 & 1,000\tabularnewline
\hline 
\# rel. doc. (ZS) & 7  & 16 & 17  & 20  & 23 & 23\tabularnewline
\hline 
\# rel. doc. (SV) & 10  & 29 & 35  & 38  & 39 & 39\tabularnewline
\hline 
pval & 0.06  & $1.1\times10^{-4}$ & $2.5\times10^{-5}$   & 0.002  & 0.036 & 0.039\tabularnewline
\hline 
\end{tabular}
\par\end{centering} 
\caption{MEDLINE collection: the two sample proportion tests for the 
relevant document counts obtained using Algorithm~\ref{alg:doc_zha} (``ZS'') and 
Algorithm~\ref{alg:doc_sv} (``SV'').
} 
\label{tbl:pval_medline} 
\end{table} 
} 

Note that the number of GKL vectors in Algorithm~\ref{alg:doc_gkl} is slightly 
larger than that of the singular triplets in Algorithm~\ref{alg:doc_sv}. This  
represents a ``compensation'' for the non-optimal choice of the low-rank approximation  
of $(I - U_k U_k^T) D$ adopted by Algorithm~\ref{alg:doc_gkl}.  
Further increase in the number of the GKL vectors may lead to higher retrieval accuracies,  
but the timing is accordingly affected. The latter observation is true for all of our experiments  
and is consistent with the relevant results in~\cite{ChenSaad:FiltLan}. 

In order to yet more carefully compare the retrieval accuracy results, 
in Table~\ref{tbl:pval_medline} we report numbers of relevant documents
(``\# rel. doc.'') among the $j$ top ranked (``\# top ranked doc.'') 
for the Zha--Simon (``ZS'') and the new (``SV'') approaches. 
Here the retrieval is performed from the entire collection
($n = 1,033$), after all columns of $A$ have been appended, using the approximate
SVD's generated by the two different updating modes.  

The results  in Table~\ref{tbl:pval_medline} demonstrate  that the new
approach gives noticeably larger numbers of relevant document, i.e., a
higher  precision is  obtained  at any  level~$j$.  One might debate,
however,  whether  the  difference  in  precision  is  
statistically significant  or if it is due to chance alone.

To address this issue we perform a two sample proportion test whose goal is
to determine  whether or not  the difference between two proportions 
is significant, see, e.g., \cite[Chap. 10]{Wasserman-book}.
Given the two proportions of relevant documents out of $j$ top ranked, 
the test has as its null hypothesis that the proportions are drawn from 
the same binomial
distribution. The ``pval'' values, giving the probability under the null 
hypothesis, 
are reported in Table~\ref{tbl:pval_medline}. One can observe that
these values appear to be very small, not exceeding $0.06$ (for $j = 10$)  
and getting as low as $2.5\times10^{-5}$ (for $j = 40$). Thus, the null 
hypothesis can be rejected with at least 94\% confidence for any level $j$, 
i.e., the difference in precisions can be seen as 
statistically significant.

 
\iftoggle{LONG}{ 
\begin{figure}[h] 
 \begin{center} 	 
 \includegraphics[width = 5.5cm]{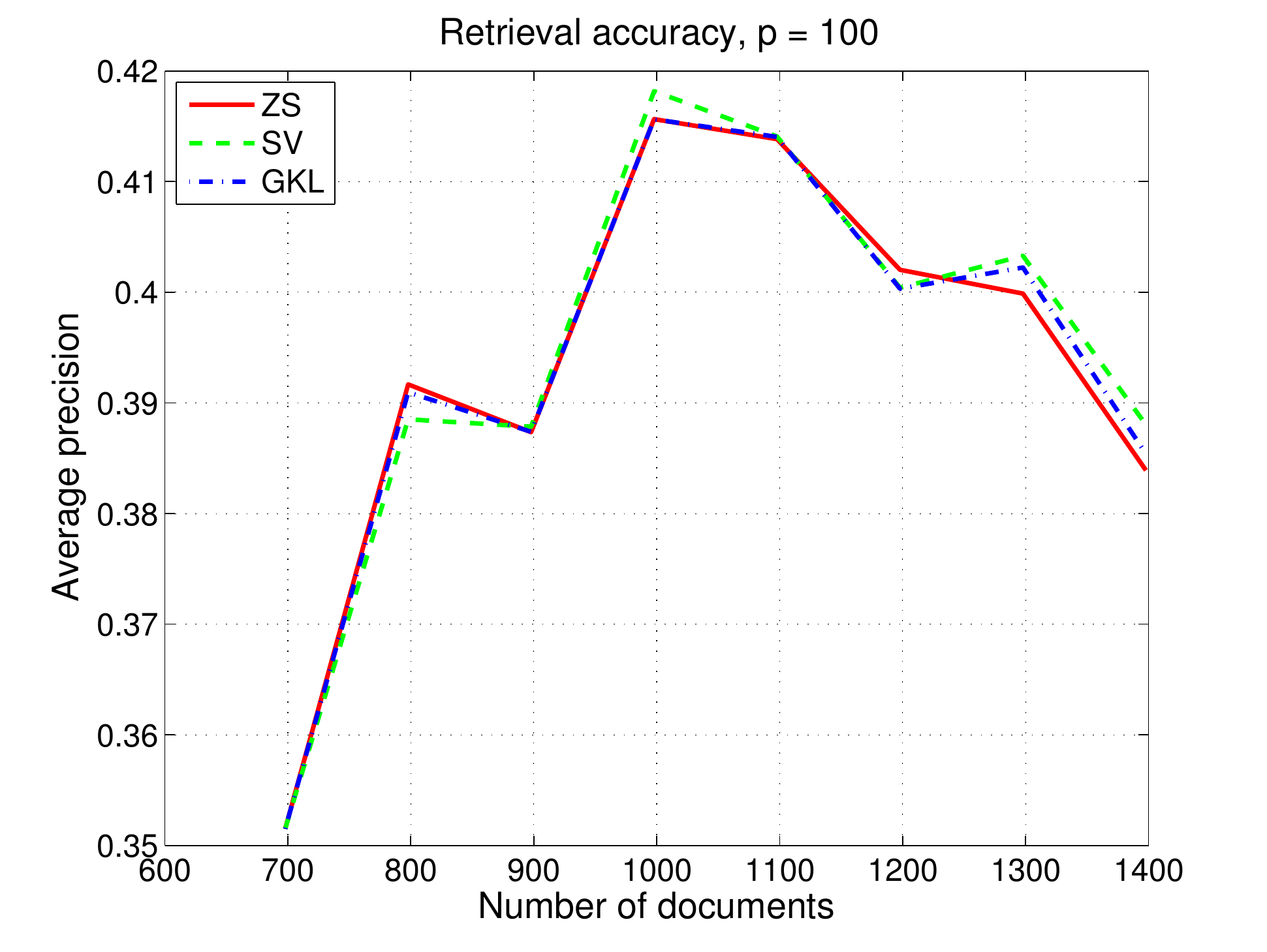} 
 \includegraphics[width = 5.5cm]{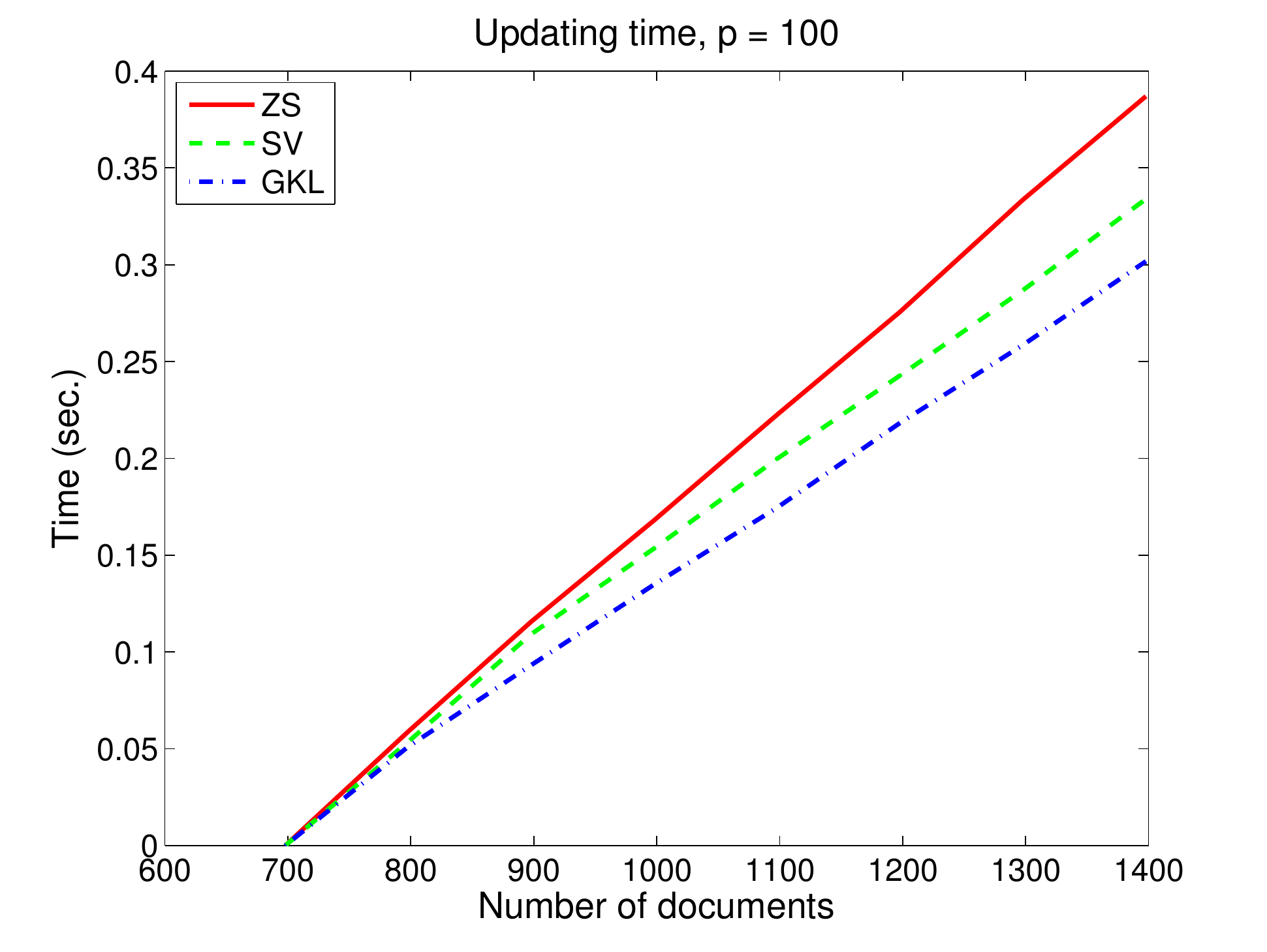} \\ 
 \includegraphics[width = 5.5cm]{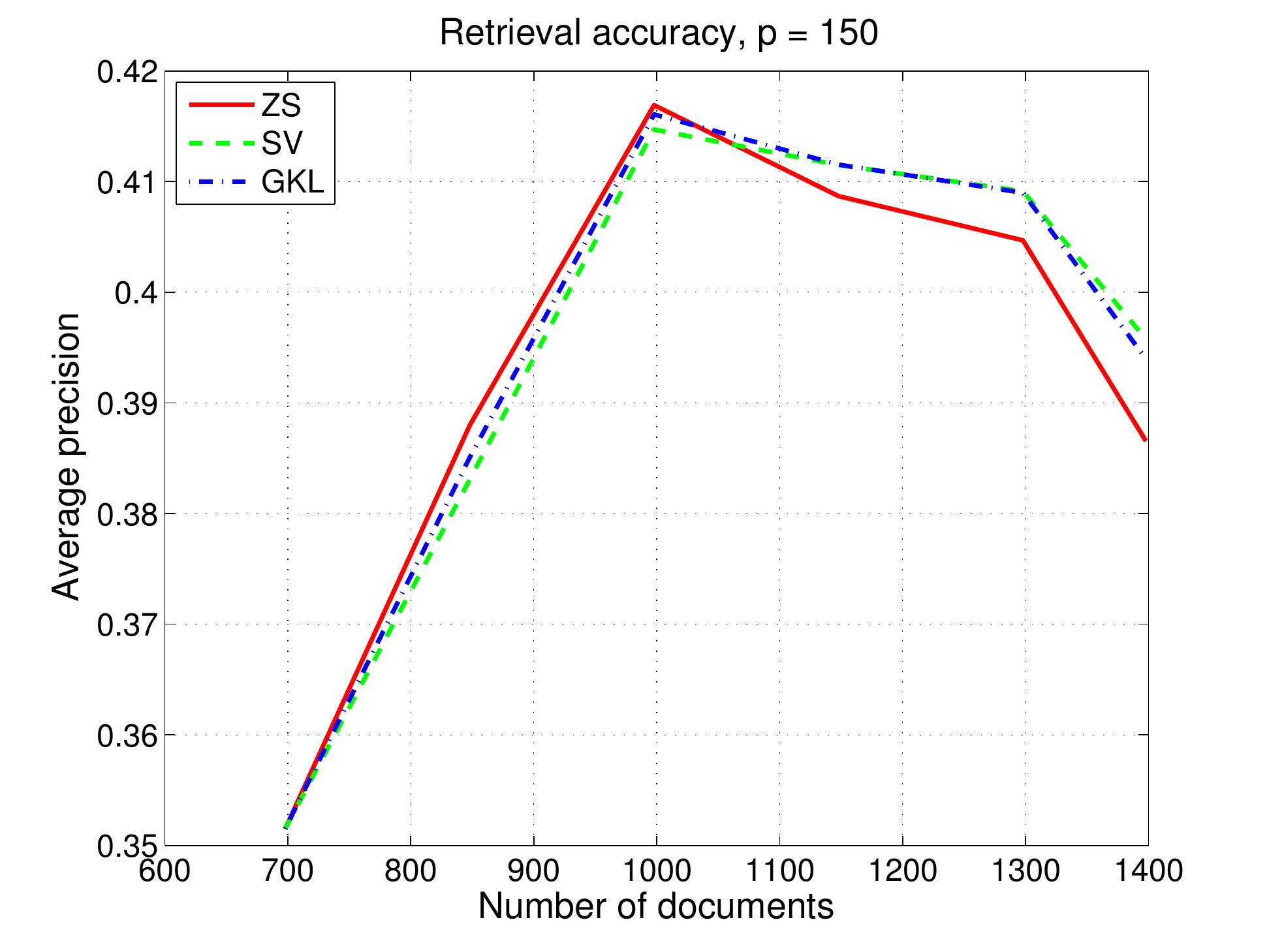} 
 \includegraphics[width = 5.5cm]{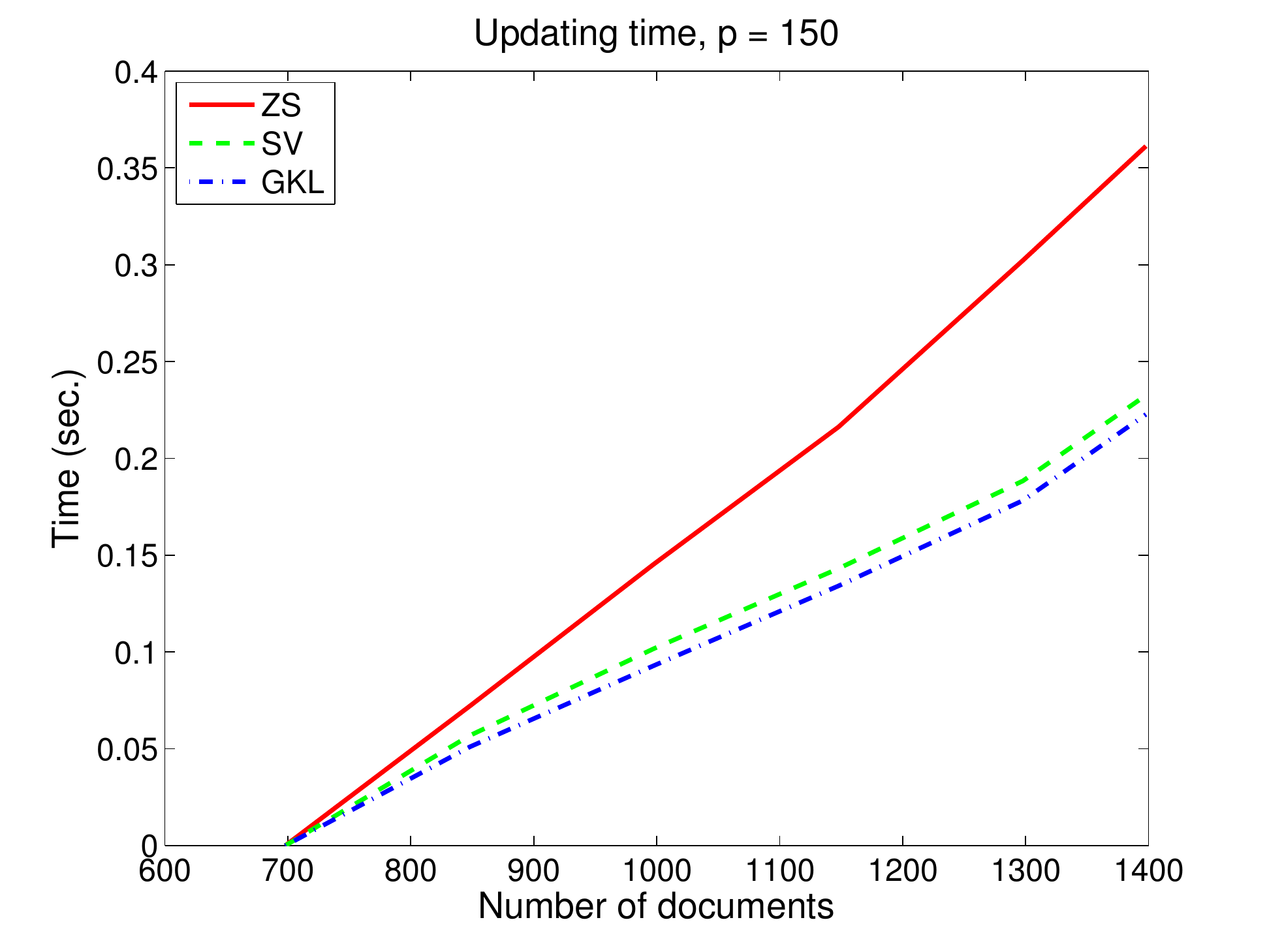} 
 \end{center} 
 \caption{ 
CRANFIELD collection: $m = 3,763$, $n = 1,398$, $k = 150$, $t = 698$, $n_q = 225$. 
The average precision and time for adding groups of $p = 100$ (top) and $p = 150$ (bottom)  
documents. The number of singular triplets of $(I - U_k U_k^T) D$ computed by Algorithm~\ref{alg:doc_sv} (``SV'') 
is $25$ (top and bottom). The number of GKL steps in Algorithm~\ref{alg:doc_gkl} (``GKL'') is 51 (top) and  
45 (bottom).  
The methods are compared to Algorithm~\ref{alg:doc_zha} (``ZS'').
}\label{fig:cran} 
\end{figure}

Figure~\ref{fig:cran} displays the results for the CRANFIELD collection. 
Similar to MEDLINE, CRANFIELD represents another example of a small text collection, 
with the term-document matrix having $m = 3,763$ rows and $n = 1,398$ columns.     
Following our test framework, we fix the initial $t = 698$ columns and add the rest in groups  
of $p = 100$ (top) and $p = 150$ (bottom). The dimension $k$ of the reduced subspace is set to $150$. 
 
The number $l$ of singular triplets   
in Algorithm~\ref{alg:doc_sv} is chosen to be 25 for both  
values of $p$. The number of GKL steps is set to 51 ($p = 100$) and 45 ($p = 150$).   
In contrast to the previous example, smaller values of $l$ fail to deliver acceptable retrieval accuracies. 
Nevertheless, as can be seen in Figure~\ref{fig:cran} (right),  
the new updating schemes are still noticeably faster than Algorithm~\ref{alg:doc_zha}.   
The retrieval accuracy is comparable for all three approaches, and is slightly higher  
for the new schemes at the later updates. 
} 
 
\begin{figure}[h] 
 \begin{center} 	 
 \includegraphics[width = 5.5cm]{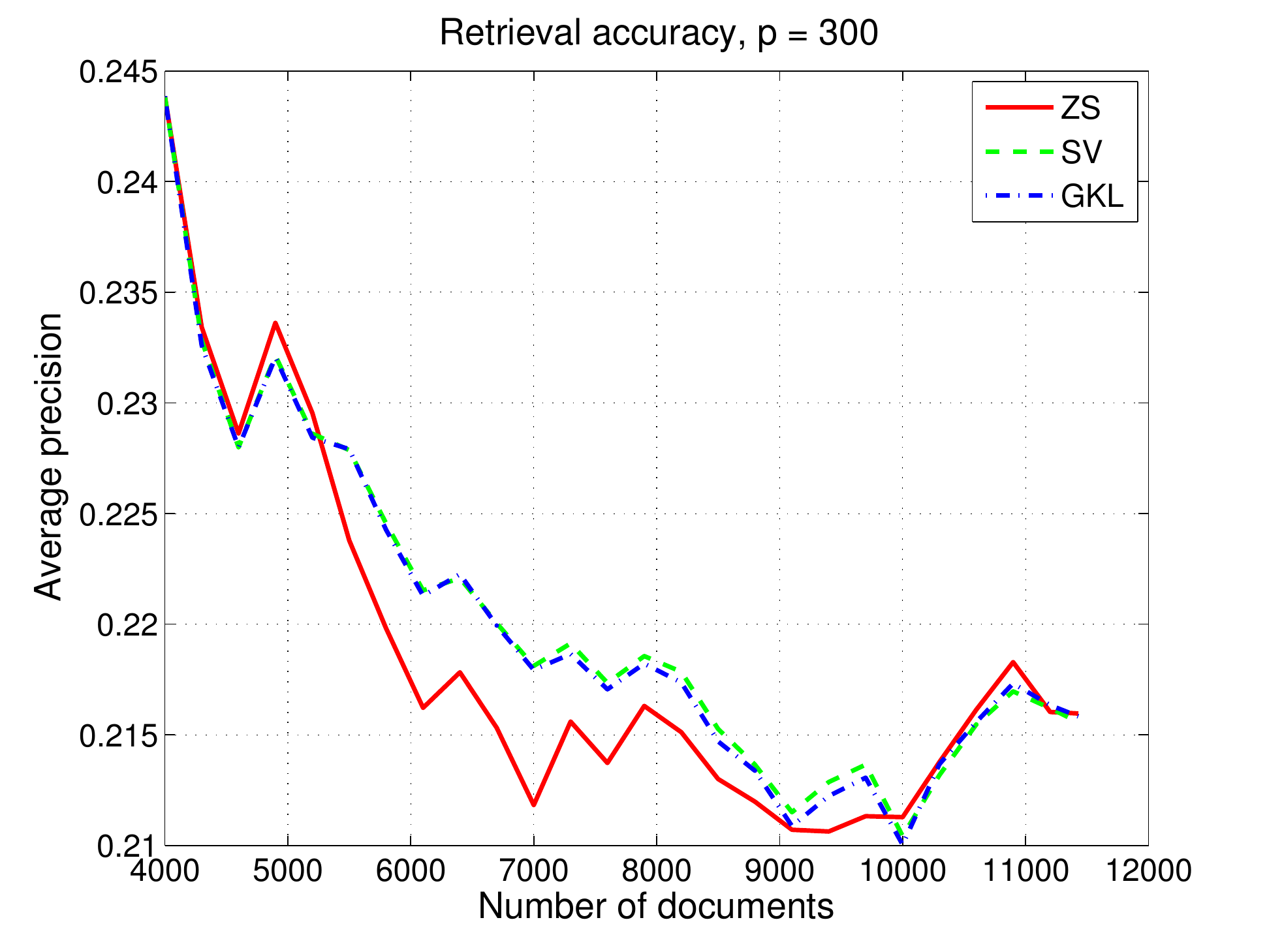} 
 \includegraphics[width = 5.5cm]{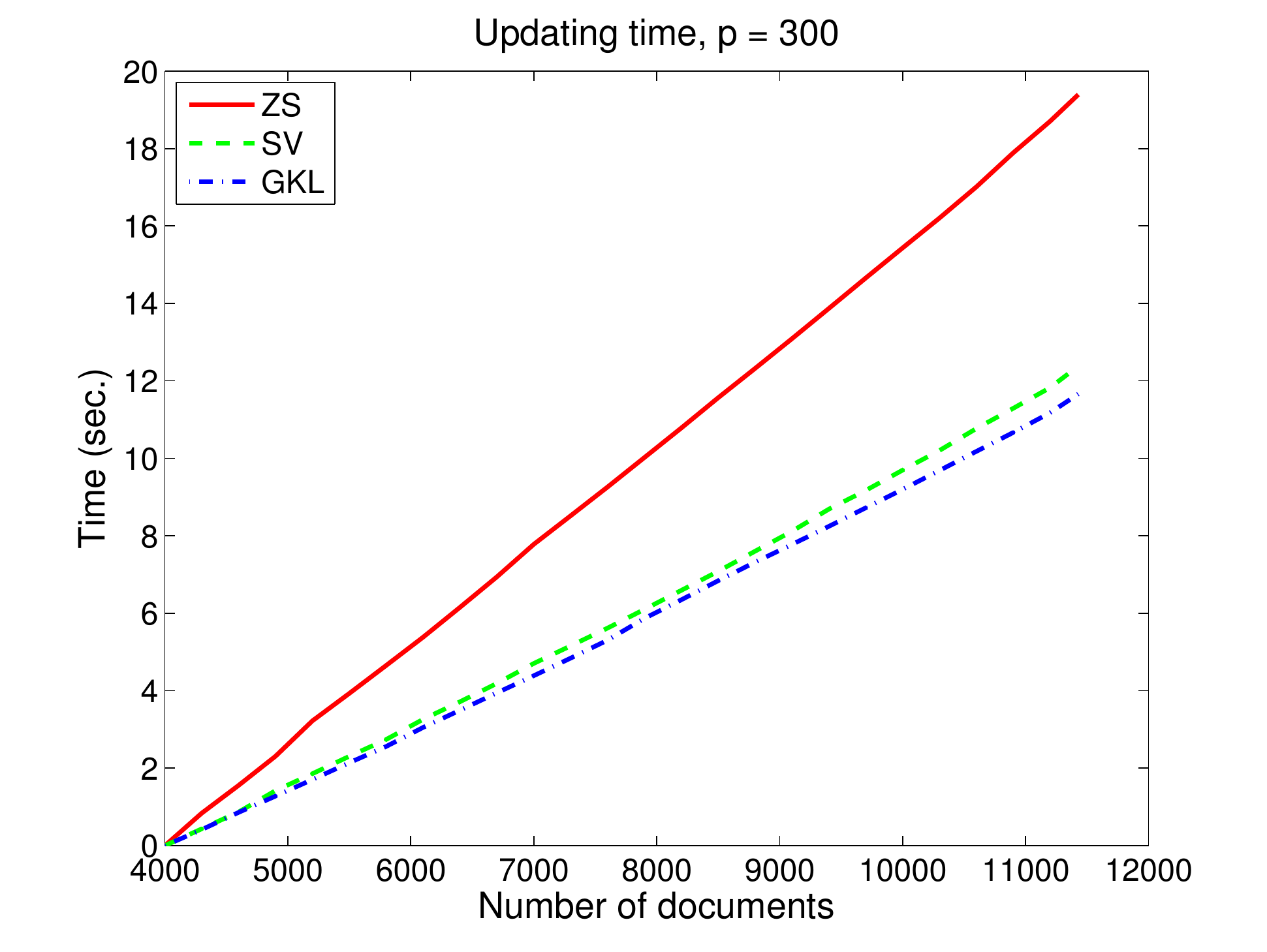} \\ 
 \includegraphics[width = 5.5cm]{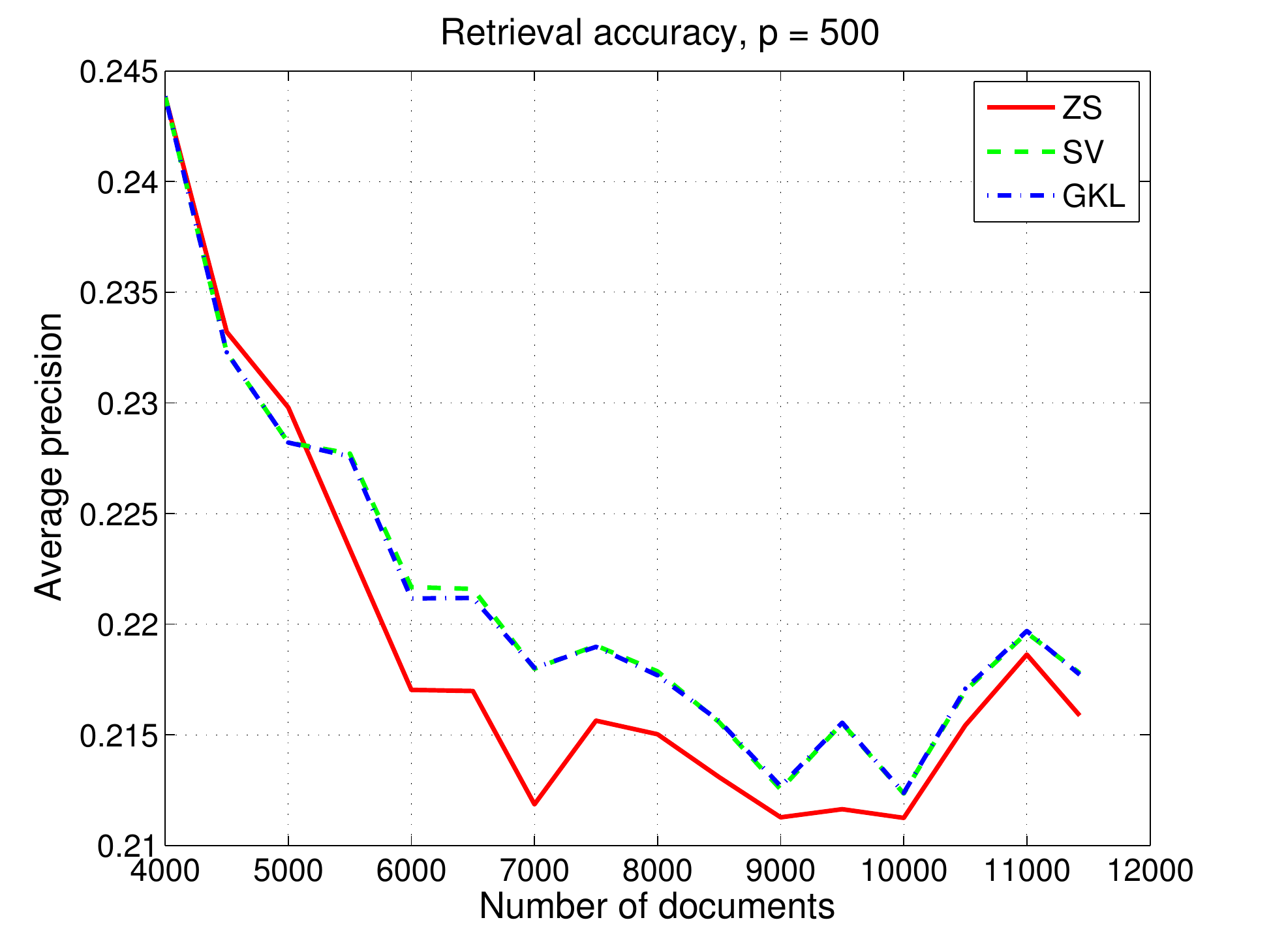} 
 \includegraphics[width = 5.5cm]{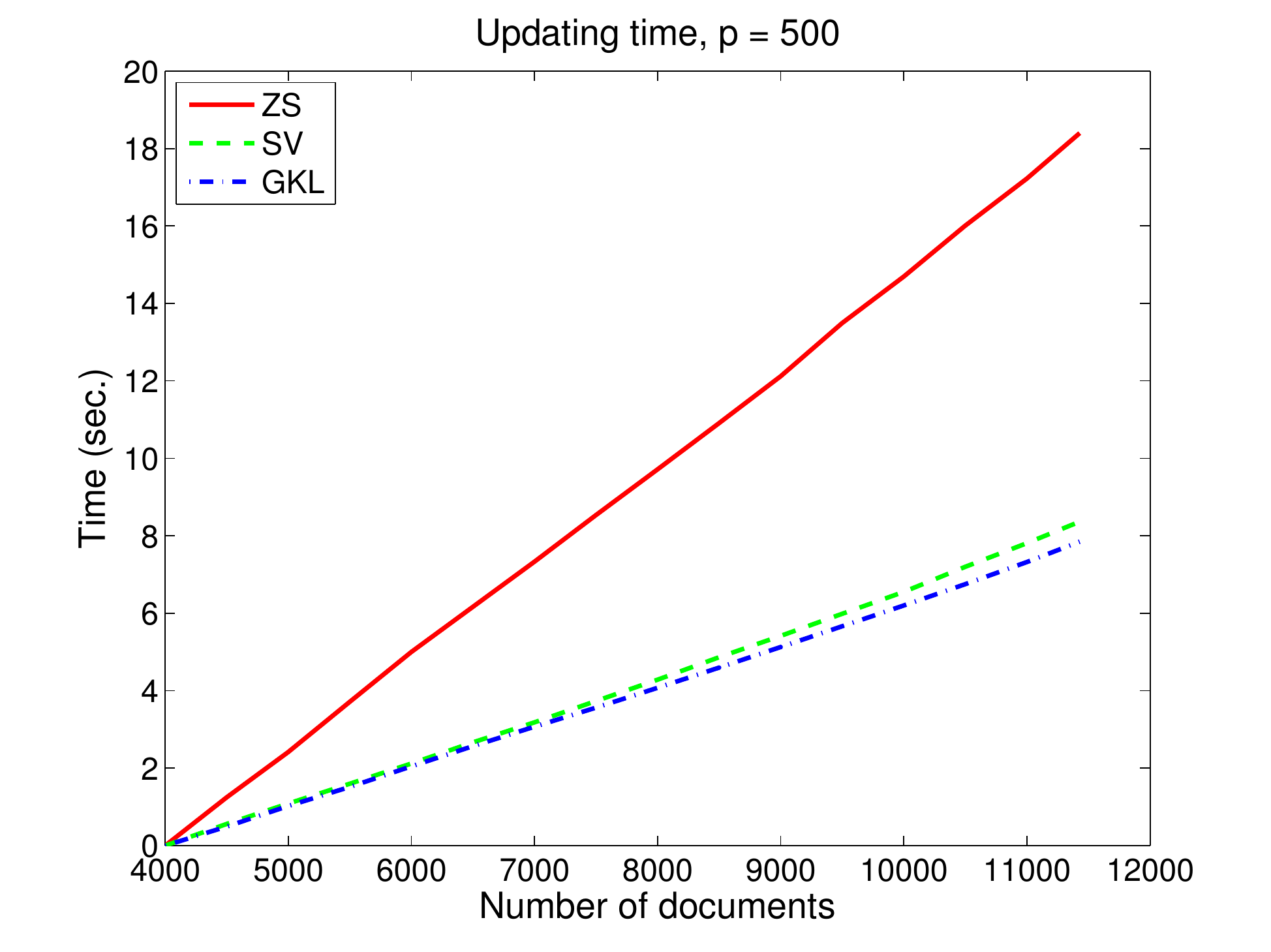} \\ 
 \end{center} 
 \caption{ 
NPL collection: $m = 7,491$, $n = 11,429$, $k = 550$, $t = 4,000$, $n_q = 93$.  
The average precision and time for adding groups of $p = 300$ (top) and $p = 500$ (bottom)  
documents. The number of singular triplets of $(I - U_k U_k^T) D$ computed by Algorithm~\ref{alg:doc_sv} (``SV'') 
is $10$ (top and bottom). The number of GKL steps in Algorithm~\ref{alg:doc_gkl} (``GKL'') is $20$ (top and bottom).   
The methods are compared to Algorithm~\ref{alg:doc_zha} (``ZS'').
}\label{fig:npl} 
\end{figure} 
 
In Figure~\ref{fig:npl} we report  results for the NPL collection. NPL 
is a  larger text collection with the  associated term-document matrix 
having  $m =  7,491$ rows  and $n  = 11,429$  columns.  Note  that, in 
contrast to the previous case,  here the number of rows is 
smaller than the number of columns. In this sense, NPL provides a more 
representative example of the real-world large-scale text collections, 
where  the number  of terms  in the  vocabulary is  limited  while the 
number of  documents, in principle, can become  arbitrarily large.  For 
the test purposes we fix $t  = 4,000$ initial columns and add the rest 
in groups of $p = 300$ and $p = 500$; $k = 550$. 

{\small 
\begin{table} 
\begin{centering} 
\begin{tabular}{|l||c|c|c|c|}
\hline 
\# top ranked doc. & 100 & 500 & 1,000 & 11,000\tabularnewline
\hline 
\# rel. doc. (ZS) & 11 & 18 & 21  & 22\tabularnewline
\hline 
\# rel. doc. (SV) & 58 & 72 & 73  & 82\tabularnewline
\hline 
pval & $2.7\times10^{-12}$ & $2.4\times10^{-9}$ & $3.9\times10^{-8}$ & $3.7\times10^{-9}$\tabularnewline
\hline 
\end{tabular}
\par\end{centering} 
\caption{NPL collection: the two sample proportion tests for the 
relevant document counts obtained using Algorithm~\ref{alg:doc_zha} (``ZS'') and 
Algorithm~\ref{alg:doc_sv} (``SV'').
} 
\label{tbl:pval_npl} 
\end{table} 
} 
 
Figure~\ref{fig:npl} shows that for the NPL dataset the new updating schemes are also 
faster and deliver a comparable retrieval accuracy.
Note that the number $l$ of the singular  
and GKL vectors is kept reasonably low.  
In particular, we request 10 singular triplets in step 1 of Algorithm~\ref{alg:doc_sv} 
and 20 GKL vectors in step 1 of Algorithm~\ref{alg:doc_gkl}.   

In Table~\ref{tbl:pval_npl}, we further assess the retrieval accuracy.
Similar to Table~\ref{tbl:pval_medline} for the MEDLINE collection, 
after completing the whole cycle of updates, we present numbers of relevant documents 
among the $j$ top ranked for the Zha--Simon (``ZS'') and the new (``SV'') approaches. 
In the same manner, we run the two sample proportion test 
for each number $j$ of the top ranked documents and report the ``pval'' values. 
As in the previous example in Table~\ref{tbl:pval_medline}, 
these values turn out to be very small, suggesting statistical significance
of the difference in the precision results. 

 
\begin{figure}[h] 
 \begin{center} 	 
 \includegraphics[width = 5.5cm]{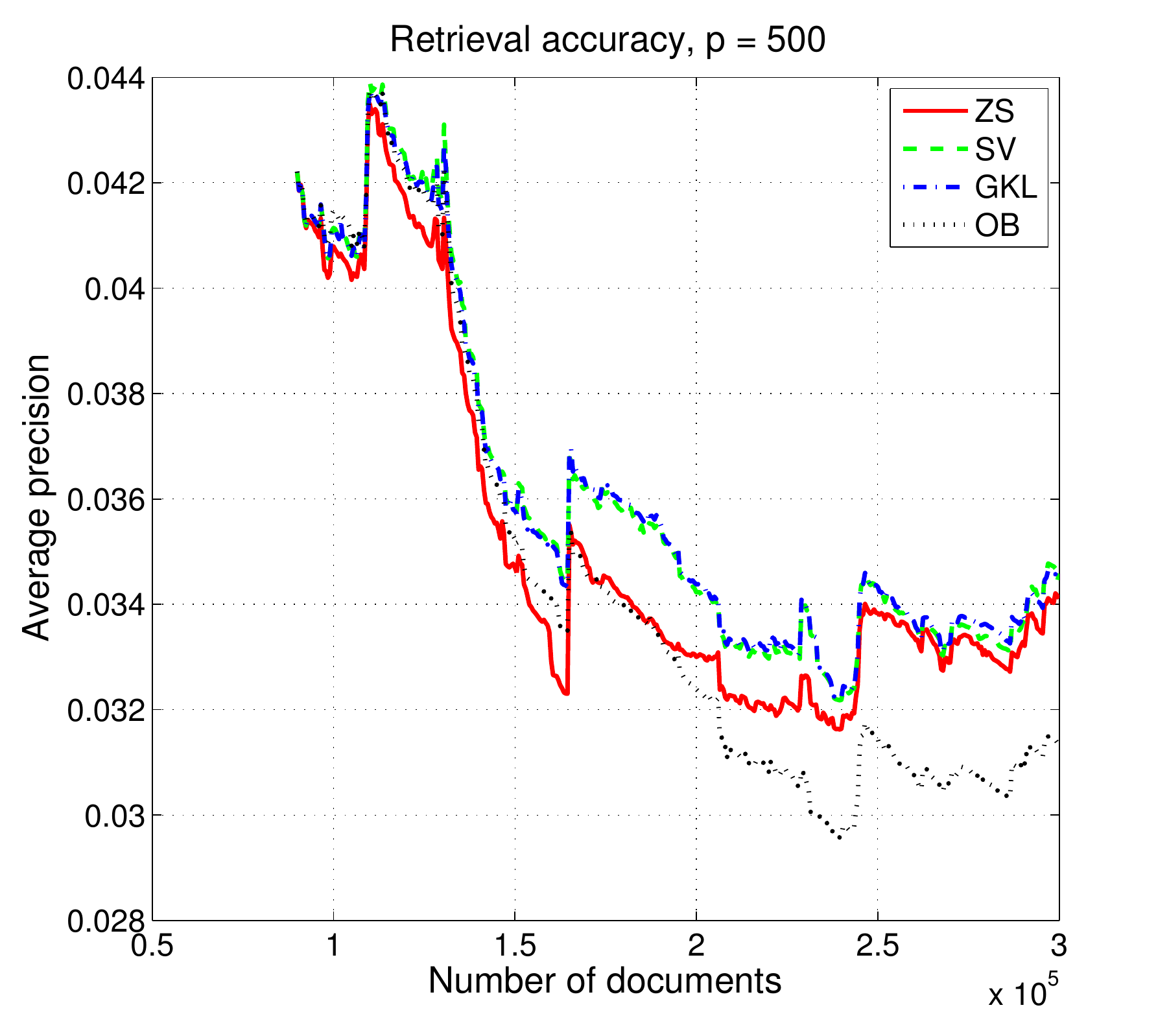} 
 \includegraphics[width = 5.5cm]{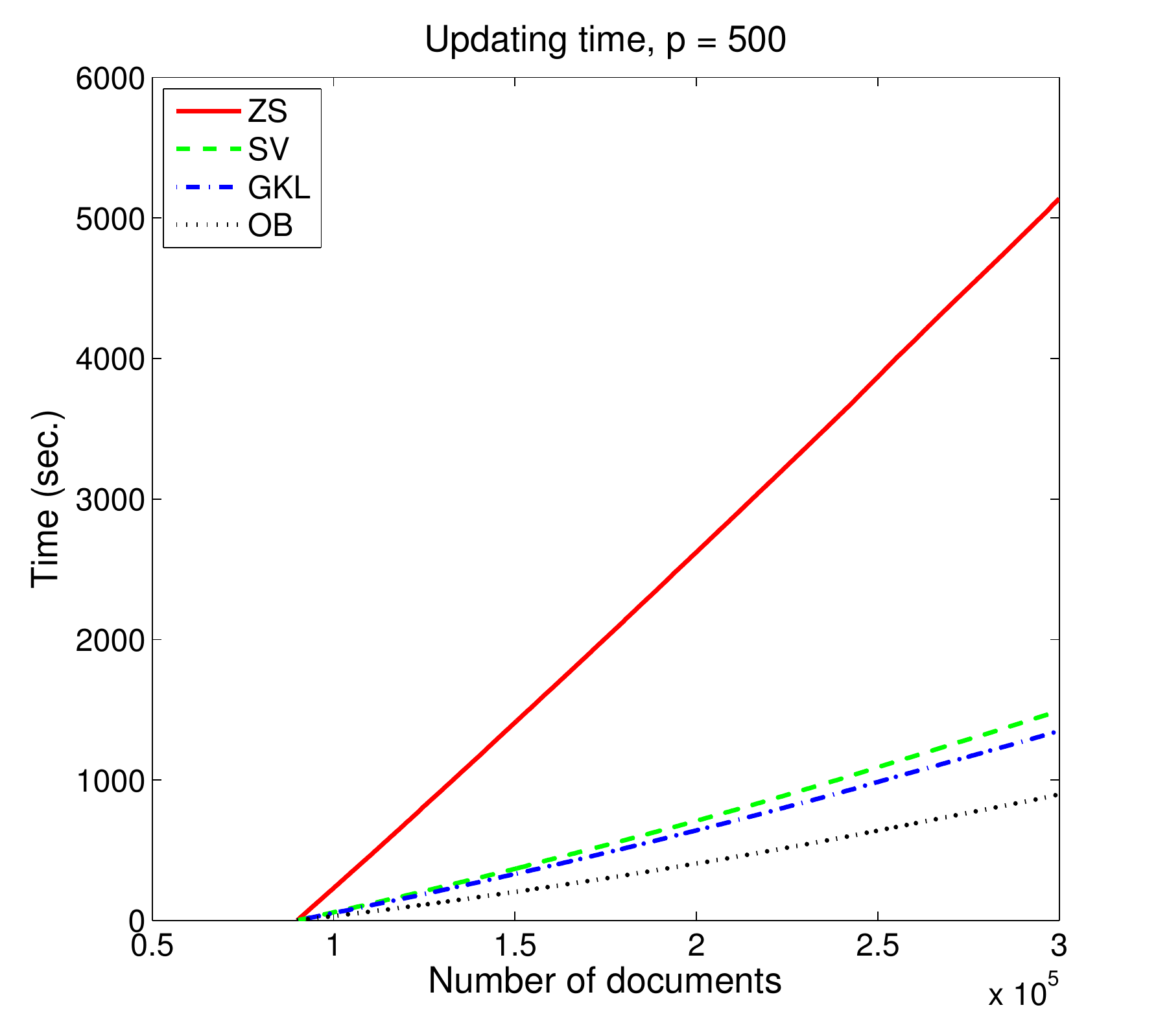} \\ 
 \includegraphics[width = 5.5cm]{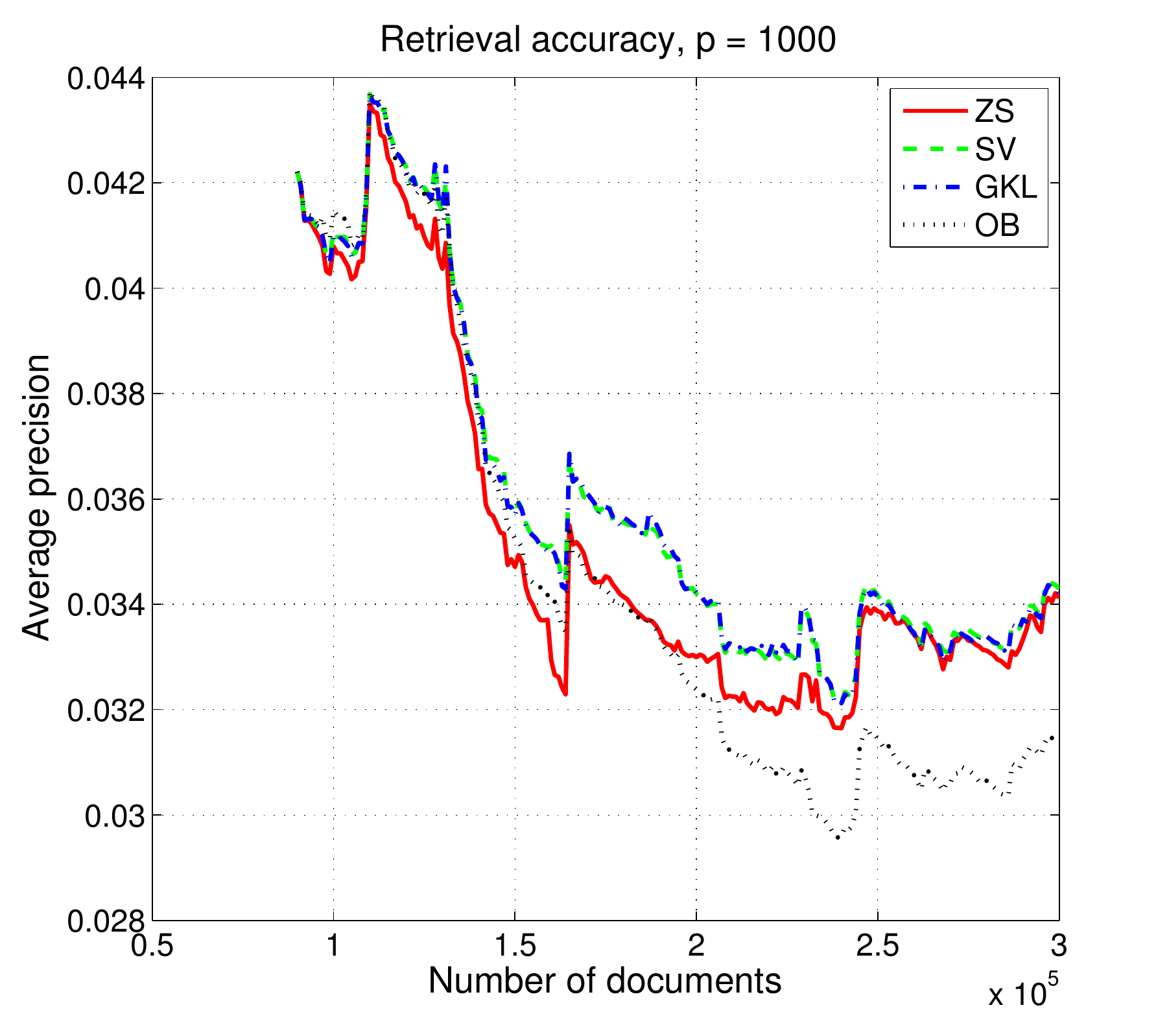} 
 \includegraphics[width = 5.5cm]{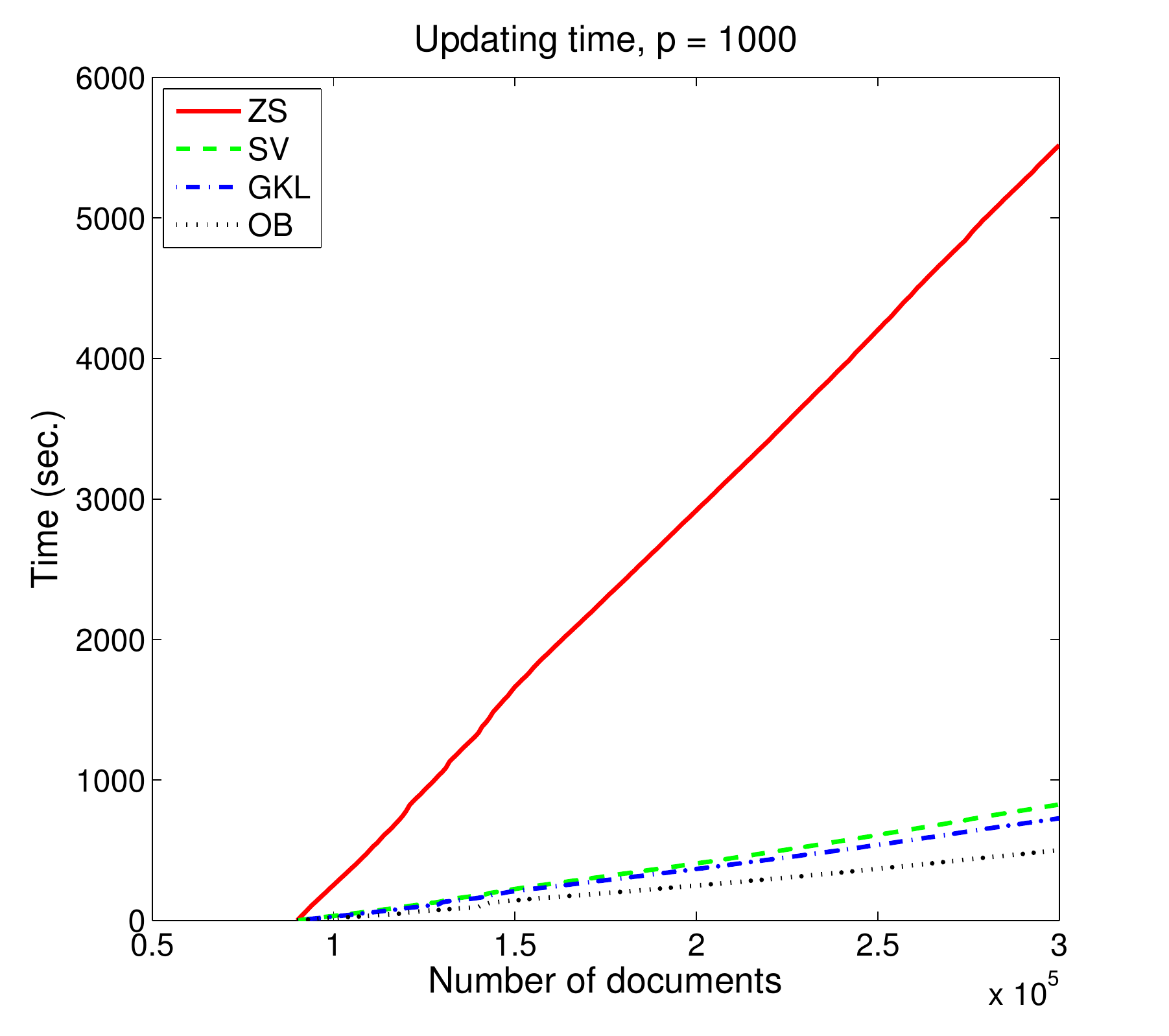}  
 \end{center} 
 \caption{ 
TREC8 collection: $m = 138,232$, $n = 528,028$, $k = 400$, $t = 90,000$, $n_q = 50$.  
The average precision and time for adding groups of $p = 500$ (top) and $p = 1000$ (bottom)  
documents. The number of singular triplets of $(I - U_k U_k^T) D$ computed by Algorithm~\ref{alg:doc_sv} (``SV'') 
is $10$. The number of GKL steps in Algorithm~\ref{alg:doc_gkl} (``GKL'') is $20$.
The methods are compared to Algorithm~\ref{alg:doc_zha} (``ZS'') and the updating 
scheme~\cite{BeDuBri:95, OBrien:94} (``OB''). 
}\label{fig:trec} 
\end{figure} 
 
Figure~\ref{fig:trec} concerns a larger example given by the TREC8 dataset. 
This dataset is known to be extensively used for testing new text mining algorithms. 
It is comprised of four document collections (\textit{Financial Times, Federal Register,  
Foreign Broadcast Information Service, and Los Angeles Times})  
from the TREC CDs 4 $\&$ 5 (copyrighted).  
The queries are from the TREC-8 \textit{ad hoc} task; see 
\url{http://trec.nist.gov/data/topics_eng/}. The relevance files  
are available at \url{http://trec.nist.gov/data/qrels_eng/}.

After the pre-processing step discussed at the beginning of this section,  
TREC8 delivers a term-document matrix with total of $m = 138,232$ 
terms and $n = 528,028$ documents. In our experiment, we fix the initial $t = 90,000$ 
columns and then incrementally add the new columns in groups of $p = 500$ (top) 
and $p = 1,000$ (bottom) until their total number reaches $300,000$.  
In both cases, the value of $l$  
is relatively small:  
we use only 10 singular triplets and 20 GKL vectors. 
  
As can be seen in~Figure~\ref{fig:trec}, the gain in the efficiency  
presented by the new schemes becomes even more pronounced 
when the methods are applied to a larger dataset.  
In particular, our tests show a three-fold speed-up of the  
overall updating process  
(420 sequential updates) for $p = 500$ and a five-fold  
speed-up (210 sequential updates) for $p=1,000$. 
Yet, in both cases, the proposed algorithms deliver a 
comparable retrieval accuracy. 

In Figure~\ref{fig:trec} we also report the results for schemes~\cite{BeDuBri:95, OBrien:94} 
(``OB'').
\footnote{The abbreviation is after the name of the author of~\cite{OBrien:94}.}
As has been previously discussed, these schemes are known to be fast but generally lack accuracy.
This is confirmed by our experiment. Remarkably, the methods proposed in this paper are essentially
as fast as those in~\cite{BeDuBri:95, OBrien:94}, but the accuracy is higher than that
of Zha--Simon schemes~\cite{Zha.Simon:99}.
Note that in contrast to~\cite{BeDuBri:95, OBrien:94}, which can be obtained by setting $l=0$ in our 
algorithms, the presence of a nonzero number $l$ of singular or GKL vectors is indeed critical 
for maintaining the accuracy.

\section{Conclusion}\label{sec:concl} 
 
This paper introduces several new algorithms for the SVD updating problem in LSI. 
The proposed schemes are based on classical projection methods 
applied to the singular value computations. A key ingredient of the  
new algorithms is the construction of low-dimensional search subspaces. 
A proper choice of such subspaces leads to fast updating schemes with a modest 
storage requirement.

In particular, we consider two options for reducing the dimensionality 
of search subspaces. The first one is based on the use of (approximate) singular vectors.  
The second options utilizes the GKL vectors. Our tests show that generally a larger number  
of GKL vectors is needed to obtain comparable retrieval accuracy.  
However, the case of singular vectors has a slightly higher computational cost.

Note that construction of 
search subspaces is not restricted only to the two techniques considered in the present work. 
Due to the established link to a Rayleigh-Ritz procedure, other approaches 
for generating search subspaces can be investigated within this framework in future research.

Our experiments demonstrate a substantial efficiency improvement over the 
state-of-the-art updating algorithms~\cite{Zha.Simon:99}.
While in our tests we have also consistently observed a slight increase 
in the retrieval accuracy, it is not 
clear how and if the observed gains are related to the proposed algorithmic
developments. This issue should be addressed in future research.
  
Since the new approach scales linearly in $p$  
(in contrast to the cubic scaling exhibited by the existing methods~\cite{Zha.Simon:99}),  
the efficiency gap becomes especially evident as $p$ increases.  
Because significantly large values of $p$ 
are likely to be encountered in the  
context of large-scale datasets, 
this means that in the future,
 algorithms such as the ones proposed in this paper,
may play a role  in reducing the cost of standard
SVD-based techniques employed in the related applications.

\bibliographystyle{siam} 
 
\bibliography{strings,IR,saad,local}

\end{document}